\documentclass[11pt]{article}%
\usepackage{amsfonts}
\usepackage{amsmath}
\usepackage{amssymb}
\usepackage{extarrows}
\usepackage{color}
\usepackage[colorlinks]{hyperref}
\usepackage[colorinlistoftodos]{todonotes}
\usepackage{graphicx}%
\providecommand{\U}[1]{\protect\rule{.1in}{.1in}}
\providecommand{\U}[1]{\protect\rule{.1in}{.1in}}
\hypersetup{colorlinks=true, linkcolor=cyan, citecolor=green,
filecolor=black, urlcolor=blue }
\newtheorem{theorem}{Theorem}

\newtheorem{definition}[theorem]{Definition}

\newtheorem{lemma}[theorem]{Lemma}

\newtheorem{proposition}[theorem]{Proposition}
\newtheorem{remark}[theorem]{Remark}

\newenvironment{proof}[1][Proof]{\noindent\textbf{#1.} }{\ \rule{0.5em}{0.5em}}

\renewcommand{\thefootnote}{\fnsymbol{footnote}}
\begin{document}

\title{Time-Delayed Generalized BSDEs}
\author{Luca Di Persio$^{a}$, Matteo Garbelli$^{a,b}$,  Lucian Maticiuc$^{c}$, Adrian Z\u{a}linescu$^{d,e}$
\bigskip\\{\small $^{a}$ Department of Computer Science, University of Verona,}\\{\small Strada le Grazie, no. 15, Verona, 37134, Italy}\\ {\small $^{b}$ Department of Mathematics, University of Trento,}\\{\small Via Sommarive 14, Povo (Trento), 38123, Italy}\\{\small $^{c}$ Faculty of Mathematics, \textquotedblleft Alexandru Ioan
Cuza\textquotedblright\ University,}\\{\small Carol I Blvd., no. 11, Ia\c{s}i, 700506, Romania}\\{\small $^{d}$ Faculty of Computer Science, \textquotedblleft Alexandru Ioan
Cuza\textquotedblright\ University,}\\{\small Carol I Blvd., no. 11, Ia\c{s}i, 700506, Romania}\\{\small $^{e}$ \textquotedblleft Octav Mayer \textquotedblright Mathematics
Institute of the Romanian Academy,}\\{\small Carol I Blvd., no. 8, Ia\c{s}i, 700506, Romania}}
\date{}
\maketitle

\begin{abstract}
We prove the existence and uniqueness of the solution of a BSDE with
time-delayed generators in the small delay setting (or equivalently small Lipschitz constant), which employs the Stieltjes integral with respect to
an increasing continuous stochastic process.  Moreover, we obtain a result of
continuity of the solution with regard to the increasing process, assuming
only uniform convergence, but not in variation. We also prove the existence for an arbitrary 
delay for a specific case
by imposing monotonicity and linearity on generators.
Lastly, we provide an application of the theoretical framework within an insurance based example.

\end{abstract}


\renewcommand{\thefootnote}{\fnsymbol{footnote}}
\footnotetext{{\scriptsize E-mail addresses: luca.dipersio@univr.it (Luca Di
Persio), lucian.maticiuc@uaic.ro (Lucian Maticiuc), adrian.zalinescu@gmail.com
(Adrian Z\u{a}linescu)}}

\textbf{AMS Classification subjects:} 60H10, 60H30\medskip

\textbf{Keywords or phrases: }generalized backward stochastic differential equations;
time--delayed generators; 
Stieltjes integral;
parameter dependence

\section{Introduction}

\label{Section 1}

Backward stochastic differential equations (BSDEs for short) were introduced
in the linear case by Bismut \cite{bi/78}, as adjoint equations involved in
the control of SDEs. The nonlinear case was considered by Pardoux and Peng
first in \cite{pa-pe/90} and then in \cite{pa-pe/92,pe/91}, where they
established a connection between BSDEs and semilinear parabolic partial
differential equations (PDEs), by the so-called nonlinear Feynman--Kac
formula. It was this kind of applications which triggered an impressive amount
of research on the subject. Concerning parabolic PDEs with Neumann boundary
conditions, Pardoux and Zhang discovered that their solutions can be linked to
BSDEs involving the integral with respect to continuous increasing processes
(Stieltjes integral).

This paper represents a first step in establishing a probabilistic
representation formula of the solutions of delayed path-dependent parabolic
PDEs with Neumann boundary conditions. It consists in studying the well
posedness of the associated BSDEs, \textit{i.e.} existence and uniqueness of
solutions, as well as stability with respect to terminal data and
coefficients. As already shown in \cite{co-di-ma-za/20} for the case of such
PDEs considered on the whole space, the generator of the associated BSDE has
to take into account the delayed-path of its solution. As a result, our
present work is concerned with the following BSDE:%
\begin{equation}
\left\{
\begin{array}
[c]{l}%
dY(t)=-F(t,Y(t),Z(t),Y_t,Z_t)dt-G(t,Y(t),Y_{t})dA(t)\medskip\\
\qquad\qquad+Z(t)dW(t),\quad t\in\left[  0,T\right]  ;\\
Y(T)=\xi,
\end{array}
\right.  \label{BSDE_0}%
\end{equation}
where the generators $F$ and $G$ depend also on the past of the solution
$\left(  Y,Z\right)  $. Here, if $\boldsymbol{x}:[-\delta,T]\rightarrow
\mathbb{R}^{n}$ is a function and $t\in\lbrack0,T],$ $\boldsymbol{x}%
_{t}:[-\delta,0]\rightarrow\mathbb{R}^{n}$ denotes the delayed-path of
$\boldsymbol{x}$, defined as%
\[
\boldsymbol{x}_{t}(\theta):=\boldsymbol{x}(t+\theta),\ \theta\in\lbrack
-\delta,0],
\]
where $\delta>0$ is a fixed delay. The coefficient $A$ is a continuous real
valued increasing process.

We recall that time-delayed BSDEs were first introduced in \cite{de-im/10} and
\cite{de-im/10x} where the authors obtained the existence and uniqueness of
the solution of the time--delayed BSDE
\begin{equation}
Y\left(  t\right)  =\xi+\int_{t}^{T}f(s,Y_{s},Z_{s})ds-\int_{t}^{T}%
Z(s)dW(s),\quad0\leq t\leq T, \label{BSDE time delay_introd 2}%
\end{equation}
where%
\[
Y_{s}:=(Y(r))_{r\in\left[  0,s\right]  }\quad\text{and}\quad Z_{s}%
:=(Z(r))_{r\in\left[  0,s\right]  }\,.
\]
In particular, the aforementioned existence and uniqueness result holds true
if the time horizon $T$ or the Lipschitz constant for the generator $f$ are
sufficiently small.

The motivation behind the introduction of a driving force $dA$ and the corresponding integral goes  beyond the link with PDE and can be traced in actuarial applications since \cite{DeFine, stev}. In the context of insurance, a BSDE such as the one described Equation \eqref{BSDE_0} can be used to model the evolution of an hedging strategy for insurance portfolio over time. In this framework the Riemann-Stieltjes  integral is linked to the sum of claims with respect to an increasing continuous process that models the cumulative distribution of claims.

This paper is organized as follows. In the remaining of this section, we
introduce the notations and set the framework of our problem. In section \ref{sec2} we
derive a result of existence and uniqueness for small delay (or small Lipschitz contant) for BSDE (\ref{BSDE_0}), based on
Banach's fixed point theorem, expressed in Theorem \ref{theorem 2}. Moreover, we provide in Proposition \ref{prop}, the well-posedness result for an arbitrary delay for a specific case assuming monotone (in the delayed term) and linear coefficients.   Section \ref{sec3} is devoted to the problem of stability
of solutions with respect to terminal data $\xi$ and coefficients $F$, $G$ and
$A$. Lastly, in Section \ref{application}, we present an insurance application dealing with a variable annuity investment that suits the theoretical setting. The main difficulty encountered in the article is to prove the convergence of the
solutions of the approximating BSDEs when the increasing process $A$ is
approximated uniformly, but not in variation. In order to tackle this problem,
we use a stochastic variant of Helly-Bray theorem, proved in the Appendix
section, as it may be an interesting result for use in other applications. 

\subsection{Problem setting and notations}

On the Euclidean space $\mathbb{R}^{n}$ we consider the Euclidean norm and
scalar product, denoted by $\left\vert \cdot\right\vert $ and $\left\langle
\cdot,\cdot\right\rangle $, respectively. If $n,k\in\mathbb{N}^{\ast}$,
$\mathbb{R}^{n\times k}$ denotes the space of real $n\times k$-matrices,
equipped with the Frobenius norm (the Euclidean norm when this space is
identified with $\mathbb{R}^{nk}$), denoted as well by $\left\vert
\cdot\right\vert $.

For $s<t$ , $C([s,t];\mathbb{R}^{n})$ represents the set of continuous
functions $\boldsymbol{x}:[s,t]\rightarrow\mathbb{R}^{d}$, endowed with the
$\sup$-norm: $\left\Vert \boldsymbol{x}\right\Vert _{C([s,t];\mathbb{R}^{n}%
)}:=\sup_{r\in\lbrack s,t]}\left\vert \boldsymbol{x}(r)\right\vert $;
$BV([s,t];\mathbb{R}^{n})$ denotes the set of right-continuous functions with
bounded variation $\boldsymbol{\eta}$$:[s,t]\rightarrow\mathbb{R}^{n}$,
\textit{i.e.} with finite total variation. Recall that the total variation of
$\boldsymbol{\eta}$ on $[s,t]$ is defined as%
\[
\mathrm{V}_{s}^{t}(\boldsymbol{\eta}):=\sup%
{\textstyle\sum\nolimits_{i=1}^{n}}
\left\vert \boldsymbol{\eta}(t_{i})-\boldsymbol{\eta}(t_{i-1})\right\vert ,
\]
where the $\sup$ is taken on all the partitions $s=t_{0}<t_{1}<\dots<t_{n}=t$.
The standard norm on $BV([s,t];\mathbb{R}^{n})$ is given by%
\[
\left\Vert \boldsymbol{\eta}\right\Vert _{BV([s,t];\mathbb{R}^{n}%
)}:=\left\vert \boldsymbol{\eta}(s)\right\vert +\mathrm{V}_{s}^{t}%
(\boldsymbol{\eta}).
\]
We will simply denote $C[s,t]$, $BV[s,t]$ instead of $C([s,t];\mathbb{R})$,
$BV([s,t];\mathbb{R})$, respectively.

If $\boldsymbol{x}:[s,t]\rightarrow\mathbb{R}^{n}$ is a Borel-measurable
function and $\boldsymbol{\eta}$$\in BV([s,t];\mathbb{R}^{n})$, by $\int%
_{s}^{t}\left\langle \boldsymbol{x}(r)d\boldsymbol{\eta}(r)\right\rangle $ we
denote the sum%
\[%
{\sum_{i=1}^{n}}
\int_{s}^{t}\left\langle \boldsymbol{x}_{i}(r)d\boldsymbol{\eta}%
_{i}(r)\right\rangle ,
\]
where $\boldsymbol{x}_{1}$, $\dots$, $\boldsymbol{x}_{n}$ and
$\boldsymbol{\eta}_{1}$, $\dots$, $\boldsymbol{\eta}_{n}$ are the components
of $\boldsymbol{x}$, respectively $\boldsymbol{\eta}$, in the case where the
Lebesgue-Stieltjes integrals are well-defined and the sum makes sense.

We fix now the framework of our problem, to be utilized throughout the article.

Let $T>0$ be a finite horizon of time, $d,m\in\mathbb{N}^{\ast}$ and
$\delta\in(0,T]$ a fixed time-delay. Let $\left(  \Omega,\mathcal{F}%
,\mathbb{P}\right)  $ be a complete probability space, $W$ a $d$-dimensional
Brownian motion and $\mathbb{F}=\left\{  \mathcal{F}_{t}\right\}
_{t\in\lbrack0,T]}$ the filtration generated by $W$, augmented by the
null-probability subsets of $\Omega$. The stochastic process $A:\Omega
\times\left[  0,T\right]  \rightarrow\mathbb{R}$ is an increasing $\mathbb{F}%
$-adapted process with $A(0)=0$, $\mathbb{P}$-a.s.

\begin{definition}
Let $p\geq2$ and $\beta\geq0$.

\begin{description}
\item[\textrm{(i)}] $\mathcal{S}^{p,m}$ denotes the space of continuous
$\mathbb{F}$--progressively measurable processes $Y:\Omega\times\left[
0,T\right]  \rightarrow\mathbb{R}^{m}$ such that
\[
\mathbb{E}\left[  \sup\nolimits_{0\leq s\leq T}|Y(s)|^{p}\right]  <+\infty\,.
\]

\item[\textrm{(ii)}] $\mathcal{S}_{\beta}^{p,m}$ denotes the space of
continuous $\mathbb{F}$--progressively measurable processes $Y:\Omega
\times\left[  0,T\right]  \rightarrow\mathbb{R}^{m}$ such that
\[
\mathbb{E}\left[  \sup\nolimits_{0\leq s\leq T}e^{\beta A(s)}|Y(s)|^{p}%
\right]  +\mathbb{E}\left[  \int_{0}^{T}e^{\beta A(s)}|Y(s)|^{2}dA(s)\right]
^{p/2}<+\infty\,.
\]

\item[\textrm{(iii)}] $\mathcal{H}_{\beta}^{p,m\times d}$ denotes the space of
$\mathbb{F}$--progressively measurable processes $Z:\Omega\times\left[
0,T\right]  \rightarrow\mathbb{R}^{m\times d}$ such that%
\[
\mathbb{E}\left[  \int_{0}^{T}e^{\beta A(s)}|Z(s)|^{2}ds\right]
^{p/2}<+\infty\,.
\]

\end{description}
\end{definition}

Instead of $\mathcal{H}_{0}^{p,m\times d}$ we will write $\mathcal{H}%
^{p,m\times d}$. The space $\mathcal{S}_{\beta}^{p,m}\times\mathcal{H}_{\beta
}^{p,m\times d}$ (in fact, its quotient with respect to $\mathbb{P\times P}%
dt$-a.e. equality) is naturally equipped with the following norm
\begin{multline*}
\Vert(Y,Z)\Vert_{p,\beta}^{p}=\mathbb{E}\left[  \sup\nolimits_{0\leq s\leq
T}e^{\beta A(s)}|Y(s)|^{p}\right]  +\mathbb{E}\left[  \int_{0}^{T}e^{\beta
A(s)}|Y(s)|^{2}dA(s)\right]  ^{p/2}\\
+\mathbb{E}\left[  \int_{0}^{T}e^{\beta A(s)}|Z(s)|^{2}ds\right]  ^{p/2}.
\end{multline*}

\section{Existence and uniqueness}
\label{sec2}

We consider the following BSDE%
\begin{multline}
Y(t)=\xi+\int_{t}^{T}F(s,Y(s),Z(s),Y_{s},Z_{s})ds+\int_{t}^{T}G(s,Y(s),Y_{s}%
)dA(s)\\
-\int_{t}^{T}Z(s)dW(s),\quad t\in\left[  0,T\right]  , \label{BSDE}%
\end{multline}
with $\xi\in L^{2}\left(  \Omega,\mathcal{F}_{T},\mathbb{P};\mathbb{R}%
^{m}\right)  $ and the generators $F:\Omega\times\lbrack0,T]\times
\mathbb{R}^{m}\times\mathbb{R}^{m\times d}\times L^{2}\left(  [-\delta
,0];\mathbb{R}^{m}\right)  \times L^{2}([-\delta,0];\mathbb{R}^{m\times
d})\rightarrow\mathbb{R}^{m}$, $G:\Omega\times\lbrack0,T]\times\mathbb{R}%
^{m}\times L^{2}\left(  [-\delta,0];\mathbb{R}^{m}\right)  \rightarrow
\mathbb{R}^{m}$ such that the functions $F\left(  \cdot,y,z,\hat{y},\hat
{z}\right)  $ and $G\left(  \cdot,y,\hat{y}\right)  $ are $\mathbb{F}%
$--progressively measurable, for any $\left(  y,z,\hat{y},\hat{z}\right)
\in\mathbb{R}^{m}\times\mathbb{R}^{m\times d}\times L^{2}\left(
[-\delta,0];\mathbb{R}^{m}\right)  \times L^{2}([-\delta,0];\mathbb{R}%
^{m\times d})$, respectively for any $\left(  y,\hat{y}\right)  \in
\mathbb{R}^{m}\times L^{2}\left(  [-\delta,0];\mathbb{R}^{m}\right)  .$

Recall that, for a function $\boldsymbol{x}:[-\delta,T]\rightarrow
\mathbb{R}^{n}$ and some $t\in\lbrack0,T],$ $\boldsymbol{x}_{t}:[-\delta
,0]\rightarrow\mathbb{R}^{n}$ denotes the delayed-path of $\boldsymbol{x}$,
defined as%
\[
\boldsymbol{x}_{t}(\theta):=\boldsymbol{x}(t+\theta),\ \theta\in\lbrack
-\delta,0].
\]
In order to define $Y_{s}$ and $Z_{s}$ even for $s<\delta$, we prolong by
convention, $Y$ by $Y(0)$ and $Z$ by $0$ on the negative real axis.

In what follows we present the assumptions required in this section. We
suppose that there exist constants $\beta$, $L$, $\tilde{L}>0$, bounded
progressively measurable stochastic processes $K,\tilde{K}:\Omega\times\left[
0,T\right]  \rightarrow\mathbb{R}_{+}$ and $\rho$, $\tilde{\rho}$ probability
measures on $\left(  [-\delta,0],\mathcal{B}\left(  \left[  -\delta,0\right]
\right)  \right)  $ such that:

\begin{description}
\item[$\mathrm{(A}_{0}\mathrm{)}$] $\mathbb{E}\left[  e^{\beta A(T)}\left(
1+\left\vert \xi\right\vert ^{2}\right)  \right]  <+\infty$;

\item[$\mathrm{(A}_{1}\mathrm{)}$] $\displaystyle\mathbb{E}\left[  \int%
_{0}^{T}e^{\beta A(t)}\left\vert F\left(  t,0,0,0,0\right)  \right\vert
^{2}dt+\int_{0}^{T}e^{\beta A(t)}\left\vert G\left(  t,0,0\right)  \right\vert
^{2}dA(t)\right]  <+\infty$.

\item[$\mathrm{(A}_{2}\mathrm{)}$] for any $t\in\lbrack0,T]$, $\left(
y,z\right)  ,(y^{\prime},z^{\prime})\in\mathbb{R}^{m}\times\mathbb{R}^{m\times
d}$, $\hat{y},\hat{y}^{\prime}\in L^{2}\left(  [-\delta,0];\mathbb{R}%
^{m}\right)  $ and $\hat{z},\hat{z}^{\prime}\in L^{2}([-\delta,0];\mathbb{R}%
^{m\times d})$, we have%
\[%
\begin{array}
[c]{rl}%
\left(  i\right)  & \displaystyle|F(t,y,z,\hat{y},\hat{z})-F(t,y^{\prime
},z^{\prime},\hat{y},\hat{z})|\leq L(|y-y^{\prime}|+|z-z^{\prime}|)\text{,
}\mathbb{P}\text{-a.s.;}\medskip\\
\left(  ii\right)  & \displaystyle|F(t,y,z,\hat{y},\hat{z})-F(t,y,z,\hat
{y}^{\prime},\hat{z}^{\prime})|^{2}\medskip\\
& \displaystyle\quad\leq K\left(  t\right)  \int_{-\delta}^{0}\left(
\left\vert \hat{y}(\theta)-\hat{y}^{\prime}(\theta)\right\vert ^{2}+\left\vert
\hat{z}(\theta)-\hat{z}^{\prime}(\theta)\right\vert ^{2}\right)  \rho
(d\theta)\text{, }\mathbb{P}\text{-a.s.;}%
\end{array}
\]

\item[$\mathrm{(A}_{3}\mathrm{)}$] for any $t\in\lbrack0,T]$, $y,y^{\prime}%
\in\mathbb{R}^{m}$ and $\hat{y},\hat{y}^{\prime}\in L^{2}\left(
[-\delta,0];\mathbb{R}^{m}\right)  $, we have%
\[%
\begin{array}
[c]{rl}%
\left(  i\right)  & \displaystyle|G(t,y,\hat{y})-G(t,y^{\prime},\hat{y}%
)|\leq\tilde{L}|y-y^{\prime}|\text{, }\mathbb{P}\text{-a.s.;}\medskip\\
\left(  ii\right)  & \displaystyle|G(t,y,\hat{y})-G(t,y,\hat{y}^{\prime}%
)|^{2}\leq\tilde{K}\left(  t\right)  \int_{-\delta}^{0}\left\vert \hat
{y}(\theta)-\hat{y}^{\prime}(\theta)\right\vert ^{2}\tilde{\rho}%
(d\theta)\text{, }\mathbb{P}\text{-a.s.;}%
\end{array}
\]

\end{description}




\begin{remark}
Let us underline that latter conditions differ from those 
used in \cite{de-im/10}, since we allow $T$\ to be arbitrary, but different from the delay $\delta
\in\left[  0,T\right]  $. 
This allows to separate the Lipschitz
constant $L$\ w.r.t. $\left(  y,z\right)  $\ from the Lipschitz
constant $K$\ w.r.t. $\left(  \hat{y},\hat{z}\right)  $; therefore
the restriction on the coefficients can avoid the constant $L$.
\end{remark}

\begin{remark}
\label{constraint}Existence and uniqueness of a solution
to the backward system (\ref{BSDE}) will be proved exploiting a standard Banach's fixed point argument which requires
$K$ or $\delta$ to be small enough.

More precisely, by denoting $K_{1}:=\sup\limits_{s\in\lbrack0,T]}K(s)$,
$\tilde{K}_{1}:=\sup\limits_{s\in\lbrack0,T]}\tilde{K}(s)$ and%
\[
\mathbf{\omega}_{\delta}:=\sup_{t\in\lbrack0,T-\delta]}\left(  A(t+\delta
)-A(t)\right)  ,
\]
we will assume the existence of a positive constant $c<c_{\beta,\tilde{L}%
}:=\min\left\{  \frac{\beta^{2}-8\tilde{L}^{2}}{4\beta^{2}},\frac{1}%
{584}\right\}  \ $such that

\begin{description}
\item[$\mathrm{(H}_{1}\mathrm{)}$] $K_{1}\cdot\max\left\{  1,T\right\}
\cdot\frac{e^{(8L^{2}+\frac{1}{2})\delta+\beta\mathbf{\omega}_{\delta}}%
}{4L^{2}}\leq c,\quad\mathbb{P}$-a.s.;

\item[$\mathrm{(H}_{2}\mathrm{)}$] $4\tilde{K}_{1}\cdot A(T)\cdot
\frac{e^{(8L^{2}+\frac{1}{2})\delta+\beta\mathbf{\omega}_{\delta}}}{\beta}\leq
c,\quad\mathbb{P}$-a.s.
\end{description}
\end{remark}

Our first result states existence and uniqueness of equation (\ref{BSDE}).

\begin{theorem}
\label{theorem 2}Let us assume that $\mathrm{(A}_{0}\mathrm{)}$--$\mathrm{(A}%
_{3}\mathrm{)}$ hold true and $\beta>2\sqrt{2}\tilde{L}$. If conditions
$\mathrm{(H}_{1}\mathrm{)}$ and $\mathrm{(H}_{2}\mathrm{)}$ are satisfied then
there exists a unique solution $\left(  Y,Z\right)  \in\mathcal{S}_{\beta
}^{2,m}\times\mathcal{H}_{\beta}^{2,m\times d}$ for (\ref{BSDE}).
\end{theorem}

\begin{proof}
The existence and uniqueness will be obtained by the Banach fixed point
theorem.$\medskip$

Let us consider the map $\Gamma:\mathcal{S}_{\beta}^{2,m}\times\mathcal{H}%
_{\beta}^{2,m\times d}\rightarrow\mathcal{S}_{\beta}^{2,m}\times
\mathcal{H}_{\beta}^{2,m\times d}$, defined in the following way: for $\left(
U,V\right)  \in\mathcal{S}_{\beta}^{2,m}\times\mathcal{H}_{\beta}^{2,m\times
d}$, $\Gamma\left(  U,V\right)  =\left(  Y,Z\right)  $, where the couple of
adapted processes $\left(  Y,Z\right)  $ is the solution to the equation%
\begin{multline}
Y(t)=\xi+\int_{t}^{T}F(s,Y(s),Z(s),U_{s},V_{s})ds+\int_{t}^{T}G(s,U(s),U_{s}%
)dA(s)\label{BSDE iterative 1}\\
-\int_{t}^{T}Z(s)dW(s),\quad t\in\left[  0,T\right]  .
\end{multline}
The existence of a unique solution $\left(  Y,Z\right)  \in\mathcal{S}%
^{2,m}\times\mathcal{H}^{2,m\times d}$ is guaranteed by \cite{pa-pe/90}.
Indeed, if we denote%
\begin{align*}
B(t)  &  :=\int_{0}^{t}G(s,U(s),U_{s})dA(s),\quad t\in\lbrack0,T];\\
\hat{F}(t,y,z)  &  :=F(t,y-B(t),z,U_{t},V_{t}),\quad t\in\lbrack0,T],\ \left(
y,z\right)  \in\mathbb{R}^{m}\times\mathbb{R}^{m\times d},
\end{align*}
then $(Y,Z)$ is a solution to equation (\ref{BSDE}) if and only if $(Y+B,Z)$
solves the equation%
\[
\hat{Y}(t)=\xi+B(T)+\int_{t}^{T}\hat{F}(s,\hat{Y}(s),Z(s),U_{s},V_{s}%
)ds-\int_{t}^{T}Z(s)dW(s),\quad t\in\left[  0,T\right]  .
\]
Since $\hat{F}$ is Lipschitz with respect to $(y,z)$, it remains to prove that
$\mathbb{E}\int_{0}^{T}\bigl|\hat{F}(t,0,0)\bigr|^{2}dt<+\infty$ and
$\xi+B(T)\in L^{2}\left(  \Omega,\mathcal{F}_{T},\mathbb{P};\mathbb{R}%
^{m}\right)  $. We have (remember that $K_{1}:=\sup\limits_{s\in\lbrack
0,T]}K(s)$ and $\tilde{K}_{1}:=\sup\limits_{s\in\lbrack0,T]}\tilde{K}(s)$):%
\[%
\begin{array}
[c]{l}%
\displaystyle\mathbb{E}\int_{0}^{T}\bigl|\hat{F}(t,0,0)\bigr|^{2}%
dt=\mathbb{E}\int_{0}^{T}\left\vert F(t,-B(t),0,U_{t},V_{t})\right\vert
^{2}dt\leq3\mathbb{E}\int_{0}^{T}\left\vert F(t,0,0,0,0)\right\vert
^{2}dt\medskip\\
\displaystyle\quad+3L^{2}\mathbb{E}\int_{0}^{T}\left\vert B(t)\right\vert
^{2}dt+3\mathbb{E}\int_{0}^{T}K\left(  t\right)  \int_{-\delta}^{0}\left(
\left\vert U(t+\theta)\right\vert ^{2}+\left\vert V(t+\theta)\right\vert
^{2}\right)  \rho(d\theta)dt\medskip\\
\displaystyle\leq3\mathbb{E}\int_{0}^{T}\left\vert F(t,0,0,0,0)\right\vert
^{2}dt+3L^{2}\mathbb{E}\int_{0}^{T}\left\vert B(t)\right\vert ^{2}dt\medskip\\
\displaystyle\quad+3T\mathbb{E}\left[  K_{1}\sup_{t\in\lbrack0,T]}\left\vert
U(t)\right\vert ^{2}\right]  +3\mathbb{E}K_{1}\int_{0}^{T}\left\vert
V(t)\right\vert ^{2}dt.
\end{array}
\]
Since $\mathrm{(A}_{1}\mathrm{)}$ holds and $K_{1}$ is bounded, we only have
to show that $\mathbb{E}\int_{0}^{T}\left\vert B(t)\right\vert ^{2}dt<+\infty$
and $\mathbb{E}\left\vert B(T)\right\vert ^{2}<+\infty$. We have%
\[%
\begin{array}
[c]{l}%
\displaystyle\mathbb{E}\int_{0}^{T}\left\vert \int_{0}^{t}G(s,U(s),U_{s}%
)dA(s)\right\vert ^{2}dt\medskip\\
\displaystyle\leq\mathbb{E}\int_{0}^{T}\left[  \int_{0}^{t}e^{\beta
A(s)}\left\vert G(s,U(s),U_{s})\right\vert ^{2}dA(s)\cdot\int_{0}^{t}e^{-\beta
A(s)}dA(s)\right]  dt\medskip\\
\displaystyle\leq\frac{T}{\beta}\mathbb{E}\int_{0}^{T}e^{\beta A(t)}\left\vert
G(t,U(t),U_{t})\right\vert ^{2}dA(t)\leq\frac{2T}{\beta}\mathbb{E}\int_{0}%
^{T}e^{\beta A(t)}\left\vert G(t,0,0)\right\vert ^{2}dA(t)\medskip\\
\displaystyle\quad+\frac{2T}{\beta}\mathbb{E}\int_{0}^{T}e^{\beta A(t)}%
L^{2}\left\vert U(t)\right\vert ^{2}dA(t)+\frac{2T}{\beta}\mathbb{E}\int%
_{0}^{T}e^{\beta A(t)}\tilde{K}\left(  t\right)  \int_{-\delta}^{0}\left\vert
U(t+\theta)\right\vert ^{2}\tilde{\rho}(d\theta)dA(t)\medskip\\
\displaystyle\leq\frac{2T}{\beta}\mathbb{E}\int_{0}^{T}e^{\beta A(t)}%
\left\vert G(t,0,0)\right\vert ^{2}dA(t)+\frac{2TL^{2}}{\beta}\mathbb{E}%
\int_{0}^{T}e^{\beta A(t)}\left\vert U(t)\right\vert ^{2}dA(t)\medskip\\
\displaystyle+\frac{2T}{\beta}\mathbb{E}\tilde{K}_{1}A(T)e^{\beta
\mathbf{\omega}_{\delta}}\sup_{t\in\lbrack0,T]}e^{\beta A(t)}\left\vert
U(t)\right\vert ^{2}<+\infty,
\end{array}
\]
by $\mathrm{(A}_{1}\mathrm{)}$ and $\mathrm{(H}_{2}\mathrm{)}$, which proves
the claim (along the way we have also proven that $\mathbb{E}\left\vert
B(T)\right\vert ^{2}<+\infty$).$\medskip$

The proof that $\left(  Y,Z\right)  \in\mathcal{S}_{\beta}^{2,m}%
\times\mathcal{H}_{\beta}^{2,m\times d}$ is very similar to that of
Proposition 1.1 from \cite{pa-zh/98}, so it is left to the reader.$\medskip$

Let us prove that $\Gamma$ is a contraction with respect to the equivalent
norm%
\begin{multline*}
\left\Vert (Y,Z)\right\Vert _{2,\alpha,\beta,a,b}^{2}:=\mathbb{E}%
\big(\sup\nolimits_{t\in\left[  0,T\right]  }e^{\alpha t+\beta A\left(
t\right)  }|Y\left(  t\right)  |^{2}\big)+a\mathbb{E}\int_{0}^{T}e^{\alpha
s+\beta A\left(  s\right)  }|Y\left(  s\right)  |^{2}dA\left(  s\right) \\
+b\mathbb{E}\int_{0}^{T}e^{\alpha s+\beta A\left(  s\right)  }|Z\left(
s\right)  |^{2}ds.
\end{multline*}
where $\alpha:=8L^{2}+\frac{1}{2}$ and the constants $a,b>0$ are yet to be chosen.

Let us consider $\left(  U^{1},V^{1}\right)  ,\left(  U^{2},V^{2}\right)
\in\mathcal{S}_{\beta}^{2,m}\times\mathcal{H}_{\beta}^{2,m\times d}$ and
$\left(  Y^{1},Z^{1}\right)  :=\Gamma\left(  U^{1},V^{1}\right)  $, $\left(
Y^{2},Z^{2}\right)  :=\Gamma\left(  U^{2},V^{2}\right)  $. For the sake of
brevity, we will denote in what follows%
\[%
\begin{array}
[c]{l}%
\Delta F\left(  s\right)  :=F(s,Y^{1}\left(  s\right)  ,Z^{1}\left(  s\right)
,U_{s}^{1},V_{s}^{1})-F(s,Y^{2}\left(  s\right)  ,Z^{2}\left(  s\right)
,U_{s}^{2},V_{s}^{2}),\medskip\\
\Delta G\left(  s\right)  :=G(s,U^{1}\left(  s\right)  ,U_{s}^{1}%
)-G(s,U^{2}\left(  s\right)  ,U_{s}^{2}),\medskip\\
\Delta U\left(  s\right)  :=U^{1}\left(  s\right)  -U^{2}\left(  s\right)
,\quad\Delta V\left(  s\right)  :=V^{1}\left(  s\right)  -V^{2}\left(
s\right)  ,\medskip\\
\Delta Y\left(  s\right)  :=Y^{1}\left(  s\right)  -Y^{2}\left(  s\right)
,\quad\Delta Z\left(  s\right)  :=Z^{1}\left(  s\right)  -Z^{2}\left(
s\right)  .
\end{array}
\]
Exploiting It\^{o}'s formula we have, for any $t\in\left[  0,T\right]  $%
\[%
\begin{array}
[c]{l}%
\displaystyle e^{\alpha t+\beta A\left(  t\right)  }|\Delta Y\left(  t\right)
|^{2}+\int_{t}^{T}e^{\alpha s+\beta A\left(  s\right)  }|\Delta Y\left(
s\right)  |^{2}\left(  \alpha ds+\beta dA\left(  s\right)  \right)  +\int%
_{t}^{T}e^{\alpha s+\beta A\left(  s\right)  }|\Delta Z\left(  s\right)
|^{2}ds\medskip\\
\displaystyle=e^{\alpha T+\beta A\left(  T\right)  }|\Delta Y\left(  T\right)
|^{2}-2\int_{t}^{T}e^{\alpha s+\beta A\left(  s\right)  }\langle\Delta
Y\left(  s\right)  ,\Delta Z\left(  s\right)  dW\left(  s\right)
\rangle\medskip\\
\displaystyle\quad+2\int_{t}^{T}e^{\alpha s+\beta A\left(  s\right)  }%
\langle\Delta Y\left(  s\right)  ,\Delta F(s)\rangle ds+2\int_{t}^{T}e^{\alpha
s+\beta A\left(  s\right)  }\langle\Delta Y\left(  s\right)  ,\Delta G\left(
s\right)  \rangle dA\left(  s\right)  \,.
\end{array}
\]
From assumptions $\mathrm{(A}_{2}\mathrm{)}$--$\mathrm{(A}_{3}\mathrm{)}$ we
obtain,%
\[%
\begin{array}
[c]{l}%
\displaystyle2\Big|\int_{t}^{T}e^{\alpha s+\beta A\left(  s\right)  }%
\langle\Delta Y\left(  s\right)  ,\Delta F(s)\rangle ds\Big|\leq2\int_{t}%
^{T}e^{\alpha s+\beta A\left(  s\right)  }\big|\langle\Delta Y\left(
s\right)  ,\Delta F(s)\rangle\big|ds\medskip\\
\displaystyle\leq8L^{2}\int_{t}^{T}e^{\alpha s+\beta A\left(  s\right)
}|\Delta Y\left(  s\right)  |^{2}ds+\frac{1}{8 L^{2}}\int_{t}%
^{T}e^{\alpha s+\beta A\left(  s\right)  }|\Delta F\left(  s\right)
|^{2}ds\medskip\\
\displaystyle\leq8L^{2}\int_{t}^{T}e^{\alpha s+\beta A\left(  s\right)
}|\Delta Y\left(  s\right)  |^{2}ds+\frac{1}{2}\int_{t}^{T}e^{\alpha s+\beta
A\left(  s\right)  }\left(  |\Delta Y\left(  s\right)  |^{2}+|\Delta Z\left(
s\right)  |^{2}\right)  ds\medskip\\
\displaystyle\quad+\frac{K_{1}T}{4L^{2}}e^{\alpha\delta+\beta\mathbf{\omega
}_{\delta}}\cdot\sup\nolimits_{s\in\left[  0,T\right]  }\big(e^{\alpha s+\beta
A\left(  s\right)  }|\Delta U\left(  s\right)  |^{2}\big)\medskip\\
\displaystyle\quad+\frac{K_{1}}{4L^{2}}e^{\alpha\delta+\beta\mathbf{\omega
}_{\delta}}\cdot\int_{0}^{T}e^{\alpha s+\beta A\left(  s\right)  }|\Delta
V\left(  s\right)  |^{2}ds
\end{array}
\]
and%
\[%
\begin{array}
[c]{l}%
\displaystyle2\Big|\int_{t}^{T}e^{\alpha s+\beta A\left(  s\right)  }%
\langle\Delta Y\left(  s\right)  ,\Delta G(s)\rangle dA\left(  s\right)
\Big|\leq2\int_{t}^{T}e^{\alpha s+\beta A\left(  s\right)  }\big|\langle\Delta
Y\left(  s\right)  ,\Delta G(s)\rangle\big|dA(s)\medskip\\
\displaystyle\leq\frac{\beta}{2}\int_{t}^{T}e^{\alpha s+\beta A\left(
s\right)  }|\Delta Y\left(  s\right)  |^{2}dA\left(  s\right)  +\frac{2}%
{\beta}\int_{t}^{T}e^{\alpha s+\beta A\left(  s\right)  }|\Delta G\left(
s\right)  |^{2}dA\left(  s\right)  \medskip\\
\displaystyle\leq\frac{\beta}{2}\int_{t}^{T}e^{\alpha s+\beta A\left(
s\right)  }|\Delta Y\left(  s\right)  |^{2}dA\left(  s\right)  +\frac
{4\tilde{L}^{2}}{\beta}\int_{t}^{T}e^{\alpha s+\beta A\left(  s\right)
}|\Delta U\left(  s\right)  |^{2}dA\left(  s\right)  \medskip\\
\displaystyle\quad+\frac{4\tilde{K}_{1}\,\,A(T)}{\beta}e^{\alpha\delta
+\beta\mathbf{\omega}_{\delta}}\cdot\sup\nolimits_{s\in\left[  0,T\right]
}\big(e^{\alpha s+\beta A\left(  s\right)  }|\Delta U\left(  s\right)
|^{2}\big).
\end{array}
\]
By $\mathrm{(H}_{1}\mathrm{)}$ and $\mathrm{(H}_{2}\mathrm{)}$, we have%
\begin{align*}
\Big(\tfrac{K_{1}T}{4L^{2}}+\tfrac{4\tilde{K}_{1}\,A(T)}{\beta}\Big)e^{\alpha
\delta+\beta\mathbf{\omega}_{\delta}}  &  \leq2c,\quad\mathbb{P}%
\text{-a.s.;}\\
\tfrac{K_{1}}{4L^{2}}e^{\alpha\delta+\beta\mathbf{\omega}_{\delta}}  &  \leq
c,\quad\mathbb{P}\text{-a.s,}%
\end{align*}
(recall that $\alpha:=8L^{2}+\frac{1}{2}$). Therefore,%
\begin{equation}%
\begin{array}
[c]{l}%
\displaystyle e^{\alpha t+\beta A\left(  t\right)  }|\Delta Y\left(  t\right)
|^{2}+\frac{\beta}{2}\int_{t}^{T}e^{\alpha s+\beta A\left(  s\right)  }|\Delta
Y\left(  s\right)  |^{2}dA\left(  s\right)  \medskip\\
\displaystyle\quad+\frac{1}{2}\int_{t}^{T}e^{\alpha s+\beta A\left(  s\right)
}|\Delta Z\left(  s\right)  |^{2}ds\medskip\\
\displaystyle\leq-2\int_{t}^{T}e^{\alpha s+\beta A\left(  s\right)  }%
\langle\Delta Y\left(  s\right)  ,\Delta Z\left(  s\right)  dW\left(
s\right)  \rangle+\frac{4\tilde{L}^{2}}{\beta}\int_{t}^{T}e^{\alpha s+\beta
A\left(  s\right)  }|\Delta U\left(  s\right)  |^{2}dA\left(  s\right)
\medskip\\
\displaystyle\quad+2c\sup\nolimits_{s\in\left[  0,T\right]  }\big(e^{\alpha
s+\beta A\left(  s\right)  }|\Delta U\left(  s\right)  |^{2}\big)+c\int%
_{0}^{T}e^{\alpha s+\beta A\left(  s\right)  }|\Delta V\left(  s\right)
|^{2}ds.
\end{array}
\label{estimate_DY_DZ}%
\end{equation}
Since $e^{\alpha s+\beta A\left(  s\right)  }\Delta Y\in\mathcal{S}^{2,m}$ and
$\Delta Z\in\mathcal{H}^{2,m\times d}$, one can show that%
\[
\mathbb{E}\left[  \int_{0}^{T}e^{\alpha s+\beta A\left(  s\right)  }%
\langle\Delta Y\left(  s\right)  ,\Delta Z\left(  s\right)  dW\left(
s\right)  \rangle\right]  =0,
\]
hence%
\begin{equation}%
\begin{array}
[c]{l}%
\displaystyle\frac{\beta}{2}\mathbb{E}\int_{0}^{T}e^{\alpha s+\beta A\left(
s\right)  }|\Delta Y\left(  s\right)  |^{2}dA\left(  s\right)  +\frac{1}%
{2}\mathbb{E}\int_{0}^{T}e^{\alpha s+\beta A\left(  s\right)  }|\Delta
Z\left(  s\right)  |^{2}ds\medskip\\
\displaystyle\leq\frac{4\tilde{L}^{2}}{\beta}\mathbb{E}\int_{0}^{T}e^{\alpha
s+\beta A\left(  s\right)  }|\Delta U\left(  s\right)  |^{2}dA\left(
s\right)  +2c\mathbb{E}\left[  \sup\nolimits_{s\in\left[  0,T\right]
}\big(e^{\alpha s+\beta A\left(  s\right)  }|\Delta U\left(  s\right)
|^{2}\big)\right]  \medskip\\
\displaystyle\quad+c\mathbb{E}\left[  \int_{0}^{T}e^{\alpha s+\beta A\left(
s\right)  }|\Delta V\left(  s\right)  |^{2}ds\right]  .
\end{array}
\label{estimate_DY_DZ2}%
\end{equation}
On the other hand, by Burkholder--Davis--Gundy's inequality, we have%
\begin{multline*}
2\mathbb{E}\bigg[\sup\nolimits_{t\in\left[  0,T\right]  }\Big|\int_{t}%
^{T}e^{\alpha s+\beta A\left(  s\right)  }\langle\Delta Y\left(  s\right)
,\Delta Z\left(  s\right)  \rangle dW\left(  s\right)  \Big|\bigg]\\
\leq\frac{1}{2}\mathbb{E}\big(\sup\nolimits_{t\in\left[  0,T\right]
}e^{\alpha t+\beta A\left(  t\right)  }|\Delta Y\left(  t\right)
|^{2}\big)+72\,\mathbb{E}\int_{0}^{T}e^{\alpha s+\beta A\left(  s\right)
}|\Delta Z\left(  s\right)  |^{2}ds.
\end{multline*}
Hence, by (\ref{estimate_DY_DZ}),%
\[%
\begin{array}
[c]{l}%
\displaystyle\frac{1}{2}\mathbb{E}\big(\sup\nolimits_{t\in\left[  0,T\right]
}e^{\alpha t+\beta A\left(  t\right)  }|\Delta Y\left(  t\right)
|^{2}\big)\medskip\\
\displaystyle\leq72\,\mathbb{E}\int_{0}^{T}e^{\alpha s+\beta A\left(
s\right)  }|\Delta Z\left(  s\right)  |^{2}ds+\frac{4\tilde{L}^{2}}{\beta
}\mathbb{E}\int_{0}^{T}e^{\alpha s+\beta A\left(  s\right)  }|\Delta U\left(
s\right)  |^{2}dA\left(  s\right)  \medskip\\
\displaystyle\quad+2c\mathbb{E}\left[  \sup\nolimits_{s\in\left[  0,T\right]
}\big(e^{\alpha s+\beta A\left(  s\right)  }|\Delta U\left(  s\right)
|^{2}\big)\right]  +c\mathbb{E}\left[  \int_{0}^{T}e^{\alpha s+\beta A\left(
s\right)  }|\Delta V\left(  s\right)  |^{2}ds\right]  .
\end{array}
\]
Thus, with $a:=\frac{\lambda\beta}{2}$, $b:=\frac{\lambda}{2}-144$ and some
$\lambda>288$, by taking into account (\ref{estimate_DY_DZ2}), we obtain%
\[%
\begin{array}
[c]{l}%
\displaystyle\mathbb{E}\big(\sup\nolimits_{t\in\left[  0,T\right]  }e^{\alpha
t+\beta A\left(  t\right)  }|\Delta Y\left(  t\right)  |^{2}\big)+a\int%
_{0}^{T}e^{\alpha s+\beta A\left(  s\right)  }|\Delta Y\left(  s\right)
|^{2}dA\left(  s\right)  \medskip\\
\displaystyle\quad+b\mathbb{E}\int_{0}^{T}e^{\alpha s+\beta A\left(  s\right)
}|\Delta Z\left(  s\right)  |^{2}ds\medskip\\
\displaystyle\leq2c(2+\lambda)\,\mathbb{E}\left[  \sup\nolimits_{s\in\left[
0,T\right]  }\big(e^{\alpha s+\beta A\left(  s\right)  }|\Delta U\left(
s\right)  |^{2}\big)\right]  +\frac{4\tilde{L}^{2}}{\beta}(2+\lambda
)\mathbb{E}\int_{0}^{T}e^{\alpha s+\beta A\left(  s\right)  }|\Delta U\left(
s\right)  |^{2}dA\left(  s\right)  \medskip\\
\displaystyle\quad+c(2+\lambda)\,\mathbb{E}\int_{0}^{T}e^{\alpha r+\beta
A\left(  r\right)  }|\Delta V\left(  r\right)  |^{2}dr,
\end{array}
\]

so
\[
\left\Vert (\Delta Y,\Delta Z)\right\Vert _{2,\alpha,\beta,a,b}^{2}\leq
\mu_{\lambda}\left\Vert (\Delta U,\Delta V)\right\Vert _{2,\alpha,\beta
,a,b}^{2},
\]
where%
\[
\mu_{\lambda}:=\max\left\{  c(2+\lambda),\tfrac{8\tilde{L}^{2}(2+\lambda
)}{\lambda\beta^{2}},\tfrac{2c(2+\lambda)}{\lambda-288}\right\}  .
\]
Since $c<c_{\beta,\tilde{L}}$, we can take $\lambda$ slightly bigger than
$\frac{1}{2c_{\beta,\tilde{L}}}-2$, such that $2c(2+\lambda)<1$ and so
$\mu_{\lambda}<1$ (by the definition of $c_{\beta,\tilde{L}}$).

It follows that the application $\Gamma$ is a contraction on the Banach space
$\mathcal{S}_{\beta}^{2,m}\times\mathcal{H}_{\beta}^{2,m\times d}$. Therefore,
by Banach fixed point theorem, there exists a unique fixed point
$(Y,Z)=\Gamma(Y,Z)$ in the space $\mathcal{S}_{\beta}^{2,m}\times
\mathcal{H}_{\beta}^{2,m\times d}$, which completes our proof.\hfill$\medskip$
\end{proof}

\begin{remark}
Let us underlined that the condition on $A$ to be increasing can be relaxed assuming it to be  a continuous bounded variation $\mathbb{F}$-adapted
process with $A_{0}=0$, $\mathbb{P}$-a.s. Indeed, by considering the
increasing process $\tilde{A}(t):=\left\Vert A\right\Vert _{BV([0,t])}$,
$t\in\lbrack0,T]$ and the Radon--Nikodym derivative $\gamma(t):=\frac
{dA(t)}{d\tilde{A}(t)}$, $t\in\lbrack0,T]$, we have that $\left\vert
\gamma(t)\right\vert \leq1$, $\forall t\in\lbrack0,T]$, $\mathbb{P}$-a.s. and the 
BSDE (\ref{BSDE}) can be rewritten as%
\begin{multline*}
Y(t)=\xi+\int_{t}^{T}F(s,Y(s),Z(s),Y_{s},Z_{s})ds+\int_{t}^{T}\tilde
{G}(s,Y(s),Y_{s})d\tilde{A}(s)\\
-\int_{t}^{T}Z(s)dW(s),\quad t\in\left[  0,T\right]  ,
\end{multline*}
where the new coefficient $\tilde{G}:\Omega\times\lbrack0,T]\times
\mathbb{R}^{m}\times L^{2}\left(  [-\delta,0];\mathbb{R}^{m}\right)
\rightarrow\mathbb{R}^{m}$ is defined as $\tilde{G}(t,y,\hat{y}):=\gamma
(t)G(t,y,\hat{y})$, still satisfying the condition $\mathrm{(A}%
_{3}\mathrm{)}$, by replacing 
 $A$ with $\tilde{A}$ and with the same $\tilde{L}$ and $\tilde{K}$. 
\end{remark}

As shown in \cite{de-im/10}, conditions as $\mathrm{(H}_{1}\mathrm{)}$ and
$\mathrm{(H}_{2}\mathrm{)}$ restricting the magnitude of the delay are
necessary. However, in the same paper the authors provide some examples
($F\equiv KY(t-T)$ and $F\equiv K\int_{0}^{t}Y(s)ds$, with $K\leq0$) in which
the delay can be considered of arbitrary length. The next result is a first
attempt to get rid of the restrictive assumptions concerning the delay, by
imposing monotonicity and linearity on generators $F$ and $G$.

More precisely, we assume that $m=1$, $\xi\in L^{2}\left(  \Omega
,\mathcal{F}_{T},\mathbb{P}\right)  $ and we require $F$ and $G$ not depending on $Z_{s}$), namely $F:\Omega\times\lbrack0,T]\times
\mathbb{R}\times\mathbb{R}^{d}\times L^{2}\left(  [-\delta,0]\right)
\rightarrow\mathbb{R}$ and $G:\Omega
\times\lbrack0,T]\times\mathbb{R}\times L^{2}\left(  [-\delta,0]\right)
\rightarrow\mathbb{R}$.

Moreover, we require that:
\begin{description}
\item[$\mathrm{(D}_{1}\mathrm{)}$] $\hat{y}\mapsto F(t,y,z,\hat{y})$ and $\hat{y}\mapsto
G(t,y,\hat{y})$ are non-increasing with respect to the positive cone of
$L^{2}\left(  [-\delta,0]\right)  $ for all $\left(  t,y,z\right)  \in
\lbrack0,T]\times\mathbb{R}\times\mathbb{R}^{d}$, $\mathbb{P}$-a.s.;

\item[$\mathrm{(D}_{2}\mathrm{)}$] $F(t,y,z,\hat{y})=F_{0}(t)+F_{1}(y,z,\hat{y})$, $G(t,y,\hat{y})=G_{0}%
(t)+G_{1}(y,\hat{y})$, with $F_{1}$ and $G_{1}$ linear.

\end{description}

\noindent
Thus, the BSDE \eqref{BSDE} reduces to the following one:
\begin{multline}
\label{linear}
Y(t)=\xi+\int_{t}^{T} \left[ F_{0}(s)+F_{1}(Y(s),Z(s),Y_s)ds \right] +\int_{t}^{T} \left[ G_{0}%
(s)+G_{1}(Y(s), Y_s) \right]  dA(s)\\
-\int_{t}^{T}Z(s)dW(s),\quad t\in\left[  0,T\right]  .
\end{multline}

\begin{proposition}
\label{prop}
Assume conditions $\mathrm{(D}_{1}\mathrm{)}$, $\mathrm{(D}%
_{2}\mathrm{)}$ and $\mathrm{(A}_{0}\mathrm{)}$--$\mathrm{(A}%
_{3}\mathrm{)}$ hold. If $\beta>2\sqrt{2}\tilde{L}$, then there exists a
solution $\left(  Y,Z\right)  \in\mathcal{S}_{\beta}^{2,1}\times
\mathcal{H}_{\beta}^{2,1\times d}$ for (\ref{linear}).
\end{proposition}

\begin{proof}
As in the proof of theorem \ref{theorem 2}, we consider the map $\Gamma
:\mathcal{S}_{\beta}^{2,1}\rightarrow\mathcal{S}_{\beta}^{2,1}$, defined in
the following way: for $U\in\mathcal{S}_{\beta}^{2,1}$, $\Gamma\left(
U\right)  =Y$, where the couple of adapted processes $\left(  Y,Z\right)  $ is
the solution to the equation%
\begin{multline*}
Y(t)=\xi+\int_{t}^{T} \left[ F_{0}(s)+F_{1}(Y(s),Z(s),U_{s})ds \right] +\int_{t}^{T} \left[ G_{0}%
(s)+G_{1}(Y(s), U_{s}) \right]  dA(s)\\
-\int_{t}^{T}Z(s)dW(s),\quad t\in\left[  0,T\right]  .
\end{multline*}
Using the same type of computations as in the above proof, it is easy to see
that even without conditions $\mathrm{(H}_{1}\mathrm{)}$ and $\mathrm{(H}%
_{2}\mathrm{)}$, $\Gamma$ is still a Lipschitz-continuous function. By a
classical comparison theorem for BSDEs, if $U^{1}(t)\leq U^{2}(t)$
$\mathbb{P}dt$-a.e., then $Y^{1}(t)\leq Y^{2}(t)$, $\forall t\in\lbrack0,T]$,
$\mathbb{P}$-a.s., with $Y^{i}(t):=\Gamma(U^{i})$, $i=\overline{1,2}$. This
shows that $\Gamma$ is non-increasing with respect to the positive cone of
$\mathcal{S}_{\beta}^{2,1}$.

One can use now an argument from \cite[Theorem 2.2]{li-wa/13} to show that
there exist $\underline{U},\overline{U}\in\mathcal{S}_{\beta}^{2,1}$ such that
$\Gamma([\underline{U},\overline{U}])\subseteq\lbrack\underline{U}%
,\overline{U}]$, where $[\underline{U},\overline{U}]:=\left\{  U\in
\mathcal{S}_{\beta}^{2,1}\mid\underline{U}(t)\leq U(t)\leq\overline
{U}(t),\text{ }\mathbb{P}dt\text{-a.e.}\right\}  $. Obviously, $[\underline{U}%
,\overline{U}]$ is a closed, convex set of the Banach space $\mathcal{S}%
_{\beta}^{2,1}$.

Let $Y^{0}:=\underline{U}$ and, by recursion, $Y^{n+1}:=\Gamma(Y^{n})$. By the
monotonicity property of $\Gamma$, it is easy to show that $\forall
t\in\lbrack0,T]$, $\mathbb{P}$-a.s.,%
\[
\underline{U}(t)=Y^{0}(t)\leq Y^{2}(t)\leq\dots\leq Y^{2n}(t)\leq\dots\leq
Y^{2n+1}(t)\leq\dots\leq Y^{3}(t)\leq Y^{1}(t)\leq\overline{U}(t).
\]
Let $\underline{Y}(t):=\lim_{n\rightarrow\infty}Y^{2n}(t)$ and $\overline
{Y}(t):=\lim_{n\rightarrow\infty}Y^{2n+1}(t)$. Since $\underline{U}%
,\overline{U}\in\mathcal{S}_{\beta}^{2,1}$, for any $H\in L^{2}(\Omega
;BV[0,T])$ or $H\in L^{2}(\Omega\times\lbrack0,T],\mathbb{P}dA(\cdot))$ we
have, by the dominated convergence theorem,%
\begin{align*}
\lim_{n\rightarrow\infty}\mathbb{E}\int_{0}^{T}e^{\beta A(t)}Y^{2n}(t)H(t)dt
&  =\mathbb{E}\int_{0}^{T}e^{\beta A(t)}\underline{Y}(t)H(t)dt\text{ and}\\
\lim_{n\rightarrow\infty}\mathbb{E}\int_{0}^{T}e^{\beta A(t)}Y^{2n+1}(t)H(t)dt
&  =\mathbb{E}\int_{0}^{T}e^{\beta A(t)}\overline{Y}(t)H(t)dt.
\end{align*}
Hence $(e^{\beta A(\cdot)/2}Y^{2n})$ and $(e^{\beta A(\cdot)/2}Y^{2n+1})$ converge
weakly to $e^{\beta A(\cdot)/2}\underline{Y}$, respectively $e^{\beta A(\cdot
)/2}\overline{Y}$, in both $L^{2}(\Omega;C[0,T])$ and $L^{2}(\Omega\times
\lbrack0,T],\mathbb{P}dA(\cdot))$. By Mazur's lemma (applied two times), for any $n\in\mathbb{N}$
there are convex combinations, let us call them $\underline{Y}^{n}$ and
$\overline{Y}^{n}$, of the elements of $(Y^{2k})_{k\geq n}$, respectively
$(Y^{2k+1})_{k\geq n}$, such that $(e^{\beta A(\cdot)/2}\underline{Y}^{n})$ and
$(e^{\beta A(\cdot)/2}\overline{Y}^{n})$ converge strongly in both $L^{2}%
(\Omega;C[0,T])$ and $L^{2}(\Omega\times\lbrack0,T],\mathbb{P}dA(\cdot))$ to
$e^{\beta A(\cdot)/2}\underline{Y}$, respectively $e^{\beta A(\cdot)/2}%
\overline{Y}$. Therefore, $(\underline{Y}^{n})$ and $(\overline{Y}^{n})$
converge strongly in $\mathcal{S}_{\beta}^{2,1}$ to $\underline{Y}$,
respectively $e^{\beta A(\cdot)}\overline{Y}$; thus, $\lim_{n\rightarrow
\infty}\Gamma(\underline{Y}^{n})=\Gamma(\underline{Y})$ and $\lim
_{n\rightarrow\infty}\Gamma(\overline{Y}^{n})=\Gamma(\overline{Y})$.

On the other hand, by the linearity of $F_{1}$ and $G_{1}$, $\Gamma
(\underline{Y}^{n})$ and $\Gamma(\overline{Y}^{n})$ are convex combinations of
the elements of $(Y^{2k+1})_{k\geq n}$, respectively $(Y^{2k})_{k\geq n}$, so
$e^{\beta A(\cdot)/2}\Gamma(\underline{Y}^{n})$ and $e^{\beta A(\cdot)/2}%
\Gamma(\overline{Y}^{n})$ converge pointwisely to $e^{\beta
A(\cdot)/2}\overline{Y}$, respectively $e^{\beta A(\cdot)/2}\underline{Y}$.
Consequently, $\Gamma(\underline{Y})=\overline{Y}$ and $\Gamma(\overline
{Y})=\underline{Y}$. Then, setting $Y=\frac{1}{2}\underline{Y}+\frac{1}%
{2}\overline{Y}$, we have $\Gamma(Y)=Y$, which proves our claim.
\end{proof}

\section{Dependence on parameters}
\label{sec3}

Let us consider, for all $n\in\mathbb{N}^{\ast}$, the following BSDEs which
approximate (\ref{BSDE}):%
\begin{multline}
Y^{n}(t)=\xi^{n}+\int_{t}^{T}F^{n}(s,Y^{n}(s),Z^{n}(s),Y_{s}^{n},Z_{s}%
^{n})ds+\int_{t}^{T}G^{n}(s,Y^{n}(s),Y_{s}^{n})dA^{n}(s)\\
-\int_{t}^{T}Z^{n}(s)dW(s),\quad t\in\left[  0,T\right]  , \label{BSDE_n}%
\end{multline}

In order to unify the notations, we will sometimes denote $\varsigma^{0}$
instead of $\varsigma$, if $\varsigma$ is $\xi$, $A$, $F$, $G$,\ $Y$ or $Z$.
We suppose that the coefficients $\xi^{n}$, $A^{n}$, $F^{n}$, $G^{n}$,
$n\geq0$, satisfy conditions $\mathrm{(A}_{2}\mathrm{)}$--$\mathrm{(A}%
_{3}\mathrm{)}$, $\mathrm{(H}_{1}\mathrm{)}$, $\mathrm{(H}_{2}\mathrm{)}$ with
processes $K^{n}$, $\tilde{K}^{n}$, but the same constants $\beta$, $c$, $L$,
$\tilde{L}$. Moreover, we have to impose that $\beta>2\sqrt{2}\tilde{L}$.

We suppose that there exists $p>1$ such that

\begin{description}
\item[$\mathrm{(A}_{0}^{\prime}\mathrm{)}$] $\sup_{n\in\mathbb{N}}%
\mathbb{E}\left[  e^{p\beta A^{n}(T)}\left\vert \xi^{n}\right\vert
^{2p}\right]  <+\infty$.

\item[$\mathrm{(A}_{0}^{\prime\prime}\mathrm{)}$] $\sup_{n\in\mathbb{N}%
}\mathbb{E}\left[  e^{qA^{n}(T)}\right]  <+\infty$, for any $q>0$.

\item[$\mathrm{(A}_{1}^{\prime}\mathrm{)}$] $\sup_{n\in\mathbb{N}}%
\mathbb{E}\left[  \left(  \int_{0}^{T}e^{\beta A^{n}(t)}\left\vert
F^{n}\left(  t,0,0,0,0\right)  \right\vert ^{2}dt\right)  ^{p}+\left(
\int_{0}^{T}e^{\beta A^{n}(t)}\left\vert G^{n}\left(  t,0,0\right)
\right\vert ^{2}dA^{n}(t)\right)  ^{p}\right]  <+\infty$.
\end{description}

Under these assumptions, there exists a unique solution $\left(  Y^{n}%
,Z^{n}\right)  \in\mathcal{S}_{\beta}^{2,m}\times\mathcal{H}_{\beta
}^{2,m\times d}$ to equation (\ref{BSDE_n}). In fact, one can now prove by
standard computations that $(Y^{n},Z^{n})\in\mathcal{S}_{\beta}^{p,m}%
\times\mathcal{H}_{\beta}^{p,m\times d}$, $\forall n\in\mathbb{N}$ and%
\begin{equation}
\sup_{n\in\mathbb{N}}\Vert(Y^{n},Z^{n})\Vert_{p,\beta}<+\infty.
\label{(Yn,Zn) bounded}%
\end{equation}

Our aim is to show that if the coefficients $\left(  \xi^{n},A^{n},F^{n}%
,G^{n}\right)  $ of equation (\ref{BSDE_n}) converge to $\left(
\xi,A,F,G\right)  $, then $\left(  Y^{n},Z^{n}\right)  $ converge to $(Y,Z)$
in $\mathcal{S}^{2,m}\times\mathcal{H}^{2,m\times d}$. Let now specify in what
sense the convergence of the coefficients takes place. We define%
\begin{align*}
\Delta_{n}F  &  :=\sup\nolimits_{t\in\lbrack0,T],\ (y,z)\in\mathbb{R}%
^{m}\times\mathbb{R}^{m\times d},\ (\hat{y},\hat{z})\in L^{2}\left(
[-\delta,0];\mathbb{R}^{m}\times\mathbb{R}^{m\times d}\right)  }\left\vert
F^{n}(t,y,z,\hat{y},\hat{z})-F(t,y,z,\hat{y},\hat{z})\right\vert ;\\
\Delta_{n}G  &  :=\sup\nolimits_{t\in\lbrack0,T],\ y\in\mathbb{R}^{m}%
\times\mathbb{R}^{m\times d},\ \hat{y}\in L^{2}\left(  [-\delta,0];\mathbb{R}%
^{m}\right)  }\left\vert G^{n}(t,y,\hat{y})-G(t,y,\hat{y})\right\vert
\end{align*}
and impose

\begin{description}
\item[$\mathrm{(C}_{1}\mathrm{)}$] $\mathbb{E}\left[  \left\vert \xi^{n}%
-\xi\right\vert ^{2p}\right]  \rightarrow0$ as $n\rightarrow\infty$;

\item[$\mathrm{(C}_{2}\mathrm{)}$] $\mathbb{E}\sup_{t\in\lbrack0,T]}\left\vert
A^{n}(t)-A(t)\right\vert \rightarrow0$ as $n\rightarrow\infty$;

\item[$\mathrm{(C}_{3}\mathrm{)}$] $\mathbb{E[}\left(  \Delta_{n}F\right)
^{2p}+\left(  \Delta_{n}G\right)  ^{2p}]\rightarrow0$ as $n\rightarrow\infty$.
\end{description}

The uniform convergence from assumption $\mathrm{(C}_{3}\mathrm{)}$ can be
relaxed to an weaker type of convergence; however, we will work with this
hypothesis for the sake of keeping computations as simple as possible.

\begin{theorem}
\label{theorem 2'}Assume that the above assumptions are fulfilled. Then%
\[
\lim_{n\rightarrow\infty}\mathbb{E}\left[  \sup_{t\in\lbrack0,T]}\left\vert
Y^{n}(t)-Y(t)\right\vert ^{2}+\int_{0}^{T}\left\vert Z^{n}(t)-Z(t)\right\vert
^{2}dt\right]  =0.
\]

\end{theorem}

\begin{proof}
Let us denote for short%
\begin{align*}
\Delta_{n}Y(t)  &  :=Y^{n}(t)-Y(t)\,,\quad\Delta_{n}Z(t):=Z^{n}%
(t)-Z(t)\,;\quad\Delta_{n}\xi:=\xi^{n}(t)-\xi(t)\,\\
\mathbf{\omega}_{\delta}^{n}  &  :=\sup_{t\in\lbrack0,T-\delta]}\left(
A^{n}(t+\delta)-A^{n}(t)\right)  .
\end{align*}
Exactly as in the proof of Theorem \ref{theorem 2}, by $\mathrm{(H}%
_{1}\mathrm{)}$ and $\mathrm{(H}_{2}\mathrm{)}$, we have%
\begin{align*}
\Big(\tfrac{K_{1}T}{4L^{2}}+\tfrac{4\tilde{K}_{1}\,A(T)}{\beta}\Big)e^{\alpha
\delta+\beta\mathbf{\omega}_{\delta}}  &  \leq2c\,,\quad\mathbb{P}%
\text{-a.s.;}\\
\tfrac{K_{1}}{4L^{2}}e^{\alpha\delta+\beta\mathbf{\omega}_{\delta}}  &  \leq
c\,,\quad\mathbb{P}\text{-a.s,}%
\end{align*}
for all $n\in\mathbb{N}$, where $\alpha=8L^{2}+\frac{1}{2}$. Let us apply
It\^{o}'s formula to $e^{\alpha t+\beta A(t)}\left\vert Y^{n}%
(t)-Y(t)\right\vert ^{2}$:%
\[%
\begin{array}
[c]{l}%
\displaystyle e^{\alpha t+\beta A^{n}(t)}|\Delta_{n}Y(t)|^{2}+\int_{t}%
^{T}e^{\alpha s+\beta A^{n}(s)}|\Delta_{n}Y(s)|^{2}\left(  \alpha
\,ds+\beta\,dA^{n}(s)\right)  +\int_{t}^{T}e^{\alpha s+\beta A^{n}(s)}%
|\Delta_{n}Z(s)|^{2}ds\medskip\\
\displaystyle=e^{\alpha T+\beta A^{n}(T)}|\Delta_{n}\xi|^{2}-2\int_{t}%
^{T}e^{\alpha s+\beta A^{n}(s)}\langle\Delta_{n}Y(s),\Delta_{n}%
Z(s)dW(s)\rangle\medskip\\
\displaystyle\quad+2\int_{t}^{T}e^{\alpha s+\beta A^{n}(s)}\langle\Delta
_{n}Y(s),F^{n}(s,Y^{n}(s),Z^{n}(s),Y_{s}^{n},Z_{s}^{n})-F(s,Y(s),Z(s),Y_{s}%
,Z_{s})\rangle ds\medskip\\
\displaystyle\quad+2\int_{t}^{T}e^{\alpha s+\beta A^{n}(s)}\langle\Delta
_{n}Y\left(  s\right)  ,G^{n}(s,Y^{n}(s),Y_{s}^{n})dA^{n}(s)-G(s,Y(s),Y_{s}%
)dA(s)\rangle\,.
\end{array}
\]

From assumptions $\mathrm{(A}_{2}\mathrm{)}$--$\mathrm{(A}_{3}\mathrm{)}$ and
$\mathrm{(A}_{1}^{\prime}\mathrm{)}$, we have, with $K_{1}^{n}:=\sup
_{t\in\lbrack0,T]}K^{n}$ and $\tilde{K}_{1}^{n}:=\sup_{t\in\lbrack0,T]}%
\tilde{K}^{n}$,%
\[%
\begin{array}
[c]{l}%
\displaystyle2\int_{t}^{T}e^{\alpha s+\beta A^{n}(s)}\langle\Delta
_{n}Y(s),F^{n}(s,Y^{n}(s),Z^{n}(s),Y_{s}^{n},Z_{s}^{n})-F(s,Y(s),Z(s),Y_{s}%
,Z_{s})\rangle ds\medskip\\
\displaystyle\leq8L^{2}\int_{t}^{T}e^{\alpha s+\beta A^{n}(s)}|\Delta
_{n}Y\left(  s\right)  |^{2}ds+\frac{1}{2}\left\vert \Delta_{n}F\right\vert
^{2}\int_{t}^{T}e^{\alpha s+\beta A^{n}(s)}ds\medskip\\
\displaystyle\quad+\frac{1}{2}\int_{t}^{T}e^{\alpha s+\beta A^{n}\left(
s\right)  }\left(  |\Delta_{n}Y\left(  s\right)  |^{2}+|\Delta_{n}Z\left(
s\right)  |^{2}\right)  dr\medskip\\
\displaystyle\quad+\frac{K_{1}^{n}Te^{\alpha\delta+\beta\mathbf{\omega
}_{\delta}^{n}}}{4L^{2}}\sup\nolimits_{s\in\left[  0,T\right]  }\big(e^{\alpha
s+\beta A^{n}(s)}|\Delta_{n}Y\left(  s\right)  |^{2}\big)\medskip\\
\displaystyle\quad+\frac{K_{1}^{n}e^{\alpha\delta+\beta\mathbf{\omega}%
_{\delta}^{n}}}{4L^{2}}\,\int_{0}^{T}e^{\alpha r+\beta A^{n}\left(  r\right)
}|\Delta_{n}Z\left(  r\right)  |^{2}dr
\end{array}
\]
and, for all $b>0$,%
\[%
\begin{array}
[c]{l}%
\displaystyle2\int_{t}^{T}e^{\alpha s+\beta A^{n}(s)}\langle\Delta_{n}Y\left(
s\right)  ,G^{n}(s,Y^{n}(s),Y_{s}^{n})dA^{n}(s)-G(s,Y(s),Y_{s})dA(s)\rangle
\medskip\\
\displaystyle=2\int_{t}^{T}e^{\alpha s+\beta A^{n}(s)}\langle\Delta
_{n}Y\left(  s\right)  ,G^{n}(s,Y^{n}(s),Y_{s}^{n})-G^{n}(s,Y(s),Y_{s})\rangle
dA^{n}\left(  s\right) \\
\displaystyle\quad+2\int_{t}^{T}e^{\alpha s+\beta A^{n}(s)}\langle\Delta
_{n}Y\left(  s\right)  ,G^{n}(s,Y^{n}(s),Y_{s}^{n})-G(s,Y(s),Y_{s})\rangle
dA^{n}\left(  s\right)  \medskip\\
\displaystyle\quad+2\int_{t}^{T}e^{\alpha s+\beta A^{n}(s)}\langle\Delta
_{n}Y\left(  s\right)  ,G(s,Y(s),Y_{s})\rangle\,\left(  dA^{n}\left(
s\right)  -dA\left(  s\right)  \right)  \medskip\\
\displaystyle\leq2\int_{t}^{T}e^{\alpha s+\beta A^{n}(s)}\langle\Delta
_{n}Y\left(  s\right)  ,G(s,Y(s),Y_{s})\rangle\,\left(  dA^{n}\left(
s\right)  -dA\left(  s\right)  \right)  \medskip\\
\displaystyle\quad+b\int_{t}^{T}e^{\alpha s+\beta A^{n}(s)}|\Delta_{n}Y\left(
s\right)  |^{2}dA^{n}\left(  s\right)  +\frac{4\tilde{L}^{2}}{b}\left\vert
\Delta_{n}G\right\vert ^{2}\int_{t}^{T}e^{\alpha s+\beta A^{n}(s)}%
dA^{n}\left(  s\right) \\
\displaystyle\quad+\frac{\beta}{2}\int_{t}^{T}e^{\alpha s+\beta A^{n}%
(s)}|\Delta_{n}Y\left(  s\right)  |^{2}dA^{n}\left(  s\right)  +\frac
{4\tilde{L}^{2}}{\beta}\int_{t}^{T}e^{\alpha s+\beta A^{n}\left(  s\right)
}|\Delta_{n}Y\left(  s\right)  |^{2}dA^{n}\left(  s\right)  \medskip\\
\displaystyle\quad+\frac{4\tilde{K}_{1}^{n}\,A^{n}\left(  T\right)
e^{\alpha\delta+\beta\mathbf{\omega}_{\delta}^{n}}}{\beta}\sup\nolimits_{s\in
\left[  0,T\right]  }\big(e^{\alpha s+\beta A^{n}(s)}|\Delta_{n}Y\left(
s\right)  |^{2}\big)\,.
\end{array}
\]

Since $\alpha=4L^{2}+1$ and $\beta>2\sqrt{2}\tilde{L}$, one can choose
$b:=\frac{\beta}{2}-\frac{4\tilde{L}^{2}}{\beta}$ and so we obtain%
\[%
\begin{array}
[c]{l}%
\displaystyle e^{\alpha t+\beta A^{n}(t)}|\Delta_{n}Y(t)|^{2}+\frac{1}{2}%
\int_{t}^{T}e^{\alpha s+\beta A^{n}(s)}|\Delta_{n}Z(s)|^{2}ds\medskip\\
\displaystyle\leq e^{\alpha T+\beta A^{n}(T)}|\Delta_{n}\xi|^{2}-2\int_{t}%
^{T}e^{\alpha s+\beta A^{n}(s)}\langle\Delta_{n}Y(s),\Delta_{n}%
Z(s)dW(s)\rangle\medskip\\
\displaystyle\quad+2\int_{t}^{T}e^{\alpha s+\beta A^{n}(s)}\langle\Delta
_{n}Y\left(  s\right)  ,G(s,Y(s),Y_{s})\rangle\,\left(  dA^{n}\left(
s\right)  -dA\left(  s\right)  \right)  \medskip\\
\displaystyle\quad+\frac{1}{2}\left\vert \Delta_{n}F\right\vert ^{2}\int%
_{t}^{T}e^{\alpha s+\beta A^{n}(s)}ds+\frac{4\tilde{L}^{2}}{b}\left\vert
\Delta_{n}G\right\vert ^{2}\int_{t}^{T}e^{\alpha s+\beta A^{n}(s)}%
dA^{n}\left(  s\right) \\
\displaystyle\quad+\frac{K_{1}Te^{\alpha\delta+\beta\mathbf{\omega}_{\delta
}^{n}}}{4L^{2}}\sup\nolimits_{s\in\left[  0,T\right]  }\big(e^{\alpha s+\beta
A^{n}(s)}|\Delta_{n}Y\left(  s\right)  |^{2}\big)\medskip\\
\displaystyle\quad+\frac{K_{1}e^{\alpha\delta+\beta\mathbf{\omega}_{\delta
}^{n}}}{4L^{2}}\,\int_{0}^{T}e^{\alpha s+\beta A^{n}\left(  s\right)  }%
|\Delta_{n}Z\left(  s\right)  |^{2}ds\\
\displaystyle\quad+\frac{4\tilde{K}_{1}\,A^{n}\left(  T\right)  e^{\alpha
\delta+\beta\mathbf{\omega}_{\delta}^{n}}}{\beta}\,\sup\nolimits_{s\in\left[
0,T\right]  }\big(e^{\alpha s+\beta A^{n}(s)}|\Delta_{n}Y\left(  s\right)
|^{2}\big).
\end{array}
\]

Therefore, by conditions $\mathrm{(H}_{1}\mathrm{)}$ and $\mathrm{(H}%
_{2}\mathrm{)}$,%
\[%
\begin{array}
[c]{l}%
\displaystyle\frac{1}{2}\int_{t}^{T}e^{\alpha s+\beta A^{n}(s)}|\Delta
_{n}Z(s)|^{2}ds\medskip\\
\displaystyle\leq2\int_{t}^{T}e^{\alpha s+\beta A^{n}(s)}\langle\Delta
_{n}Y\left(  s\right)  ,G(s,Y(s),Y_{s})\rangle\,\left(  dA^{n}\left(
s\right)  -dA\left(  s\right)  \right)  \medskip\\
\displaystyle\quad+\frac{1}{2}\left\vert \Delta_{n}F\right\vert ^{2}\int%
_{t}^{T}e^{\alpha s+\beta A^{n}(s)}ds+\frac{4\tilde{L}^{2}}{b}\left\vert
\Delta_{n}G\right\vert ^{2}\int_{t}^{T}e^{\alpha s+\beta A^{n}(s)}%
dA^{n}\left(  s\right)  \medskip\\
\displaystyle\quad+2c\,\sup\nolimits_{s\in\left[  0,T\right]  }\big(e^{\alpha
s+\beta A^{n}(s)}|\Delta_{n}Y\left(  s\right)  |^{2}\big)+c\int_{0}%
^{T}e^{\alpha s+\beta A^{n}\left(  s\right)  }|\Delta_{n}Z\left(  s\right)
|^{2}ds.
\end{array}
\]

Exploiting Burkholder--Davis--Gundy's inequality, we have that
\[%
\begin{array}
[c]{l}%
\displaystyle2\mathbb{E}\Big[\sup\nolimits_{t\in\left[  0,T\right]  }%
\Big|\int_{t}^{T}e^{\alpha s+\beta A^{n}(s)}\langle\Delta_{n}Y(s),\Delta
_{n}Z(s)dW(s)\rangle\Big|\Big]\medskip\\
\displaystyle\leq\frac{1}{4}\,\mathbb{E}\big(e^{\alpha s+\beta A^{n}%
(s)}|\Delta_{n}Y\left(  s\right)  |^{2}\big)+144\,\mathbb{E}\int_{0}%
^{T}e^{\alpha s+\beta A^{n}\left(  s\right)  }|\Delta_{n}Z\left(  s\right)
|^{2}ds\,.
\end{array}
\]
As in the proof of Theorem \ref{theorem 2}, we obtain%
\[%
\begin{array}
[c]{l}%
\displaystyle\mathbb{E}\big(\sup\nolimits_{s\in\left[  t,T\right]  }e^{\alpha
s+\beta A^{n}\left(  s\right)  }|\Delta_{n}Y\left(  s\right)  |^{2}%
\big)+\mathbb{E}\int_{0}^{T}e^{\alpha s+\beta A^{n}\left(  s\right)  }%
|\Delta_{n}Z\left(  s\right)  |^{2}ds\medskip\\
\displaystyle\leq C\mathbb{E}\left[  |\Delta_{n}\xi|^{2p}+|\Delta_{n}%
F|^{2p}+|\Delta_{n}G|^{2p}\right]  \cdot\mathbb{E}e^{\beta qA^{n}(T)}%
\medskip\\
\displaystyle\quad+C\mathbb{E}\sup_{t\in\lbrack0,T]}\left\vert \int_{t}%
^{T}e^{\alpha s+\beta A^{n}(s)}\langle\Delta_{n}Y\left(  s\right)
,G(s,Y(s),Y_{s})\rangle\,\left(  dA^{n}\left(  s\right)  -dA\left(  s\right)
\right)  \right\vert .
\end{array}
\]
where $C$ is a positive constant and $q:=\frac{p}{p-1}$.

By conditions $\mathrm{(C}_{1}\mathrm{)}$ and $\mathrm{(A}_{0}^{\prime\prime
}\mathrm{)}$,%
\[
\lim_{n\rightarrow\infty}\mathbb{E}\left[  |\Delta_{n}\xi|^{2p}+|\Delta
_{n}F|^{2p}+|\Delta_{n}G|^{2p}\right]  \cdot\mathbb{E}e^{\beta qA^{n}(T)}=0.
\]
It remains to prove that%
\[
\lim_{n\rightarrow\infty}\mathbb{E}\sup_{t\in\lbrack0,T]}\left\vert \int%
_{t}^{T}X^{n}(s)dH^{n}(s)\right\vert =0,
\]
where, for $s\in\lbrack0,T]$,%
\begin{align*}
X^{n}(s)  &  :=e^{\alpha s+\beta A^{n}(s)}\langle\Delta_{n}Y\left(  s\right)
,G(s,Y(s),Y_{s})\rangle;\\
H^{n}(s)  &  :=A^{n}\left(  s\right)  -A\left(  s\right)  .
\end{align*}
One can prove that%
\[
\mathbb{E}\sup_{t\in\lbrack0,T]}\left\vert X^{n}(t)\right\vert ^{p}%
\]
is uniformly bounded (with respect to $n$), by (\ref{(Yn,Zn) bounded}).
Obviously, by $\mathrm{(A}_{0}^{\prime\prime}\mathrm{)}$,%
\[
\sup_{n\in\mathbb{N}^{\ast}}\mathbb{E}\sup_{t\in\lbrack0,T]}\left\vert
H^{n}(t)\right\vert ^{2}<+\infty.
\]
Hence, the sequence $\left(  X^{n},H^{n}\right)  _{n\in\mathbb{N}^{\ast}}$ is
tight in $C\left[  0,T\right]  ^{2}$. By Prokhorov's theorem, we can extract a
sequence, say $\left(  X^{n_{k}},H^{n_{k}}\right)  _{k\in\mathbb{N}^{\ast}}$,
convergent in distribution to some stochastic process $\left(  X,H\right)  $
with continuous paths. Since, by $\mathrm{(C}_{2}\mathrm{)}$, $\lim
_{n\rightarrow\infty}\mathbb{E}\sup_{t\in\lbrack0,T]}\left\vert H^{n}%
(t)\right\vert =0$, $H$ must be $\mathbb{P}$-a.s. equal to $0$. The condition
$\mathrm{(A}_{0}^{\prime\prime}\mathrm{)}$ also implies that $\sup
_{n\in\mathbb{N}}\mathbb{E}\left\Vert H^{n}\right\Vert _{BV[0,T]}^{a}<+\infty
$, for every $a>1$, so $\left\Vert H^{n}\right\Vert _{BV[0,T]}$ is bounded in
probability (\textit{i.e.}, it satisfies condition (\ref{condition_bv_bounded}%
)). We can now apply Proposition \ref{prop1}, proved as an auxiliary result in
the Appendix section, in order to derive the convergence in distribution to
$0$ of the process%
\[
\left(  \int_{0}^{t}X^{n}(s)dH^{n}(s)\right)  _{t\in\lbrack0,T]}.
\]
Since, for some $\nu>0$, the functional $\phi_{\nu}:C\left[  0,T\right]
\rightarrow\mathbb{R}$, defined by%
\[
\phi_{\nu}(\boldsymbol{x}):=\sup_{t\in\lbrack0,T]}\left\vert \boldsymbol{x}%
(T)-\boldsymbol{x}(t)\right\vert \wedge\nu,
\]
is bounded and continuous, it follows that%
\[
\mathbb{E}\left[  \sup_{t\in\lbrack0,T]}\left\vert \int_{t}^{T}X^{n}%
(s)dH^{n}(s)\right\vert \wedge\nu\right]  =0,
\]
for every $\nu>0$. Since, by Markov's inequality, for some $a\in(1,p)$%
\begin{multline*}
\mathbb{E}\sup_{t\in\lbrack0,T]}\left\vert \int_{t}^{T}X^{n}(s)dH^{n}%
(s)\right\vert \leq\mathbb{E}\left[  \sup_{t\in\lbrack0,T]}\left\vert \int%
_{t}^{T}X^{n}(s)dH^{n}(s)\right\vert \wedge\nu\right] \\
+\frac{1}{\nu^{a}}\mathbb{E}\left[  \sup_{t\in\lbrack0,T]}\left\vert \int%
_{t}^{T}X^{n}(s)dH^{n}(s)\right\vert ^{a}\right]
\end{multline*}
and%
\begin{align*}
\mathbb{E}\left[  \sup_{t\in\lbrack0,T]}\left\vert \int_{t}^{T}X^{n}%
(s)dH^{n}(s)\right\vert ^{a}\right]   &  \leq\mathbb{E}\left[  \left(
\sup_{t\in\lbrack0,T]}\left\vert X^{n}(t)\right\vert ^{a}\right)  \left\Vert
H^{n}\right\Vert _{BV[0,T]}^{a}\right] \\
&  \leq\left(  \mathbb{E}\left[  \sup_{t\in\lbrack0,T]}\left\vert
X^{n}(t)\right\vert ^{p}\right]  \right)  ^{\frac{a}{p}}\left(  \mathbb{E}%
\left\Vert H^{n}\right\Vert _{BV[0,T]}^{\frac{p}{a(p-a)}}\right)  ^{1-\frac
{a}{p}},
\end{align*}
it follows that%
\[
\lim_{n\rightarrow\infty}\mathbb{E}\sup_{t\in\lbrack0,T]}\left\vert \int%
_{t}^{T}X^{n}(s)dH^{n}(s)\right\vert =0,
\]
which concludes our proof.
\end{proof}


\section{Hedging a stream of payments with time-delayed GBSDE}
\label{application}

In this last section we present a risk management application for an insurance product, the so called variable annuity instrument, whose composition can be controlled
by the insurer  selecting an appropriate strategy. The composition for the underlying investment portfolio can be controlled internally by
the insurer to reduce the overall risk of the policyholder investment. Specifically, inspired by \cite{de/11}, we consider an insurance product where the policyholder withdraws some guaranteed amounts as a fraction of the maximum value of the investment and, additionally, is subjected to a continuous payment triggered by an increasing continuous process $A$ modelling the cumulative function of claims (or, e.g. of fees for the management of the wealth). At maturity the remaining value is converted into a life-time annuity with 
a guaranteed consumption rate $C$. 

\medskip

We consider a probability space $(\Omega, \mathcal{F}, \mathbb{P})$ with associated natural filtration $\mathbb{F} = (\mathcal{F}_t)_{0 \leq t \leq R}$ generated by a Brownian motion $W := (W (t), \, 0 \leq t \leq T)$ and a finite time horizon $T \leq \infty$. 

The goal of the investor is to replicate the insurance by investing into the assets and to quantify the risk of the investing activities. In the terminology of \cite{ka-pe-qu/97}, we focus on an investment composed by a  risk free asset $S_0$ and a risky asset $D$.

The price of the risk free asset $S_0 := (S_0 (t), \, 0 \leq t \leq T)$ is given by the equation

\begin{equation}
    \dfrac{dS_0(t)}{S_0 (t)} = r(t)dt, \qquad S_0 (0) = 1,
\end{equation}
where $r$ describes the risk free interest rate being a non-negative $\mathbb{F}$-progressively measurable stochastic process.

The  price of the risky bond $D := (D (t), \, 0 \leq t \leq T)$ with maturity $T$ is given by

\begin{equation}
    \dfrac{dD(t)}{D (t)} = \left( r(t) + \sigma (t) \theta(t) \right) dt + \sigma (t) dW(t), \qquad S (0) = x,
\end{equation}
where the volatility $\sigma := (\sigma (t), \, 0 \leq t \leq T)$ and the risk premium $\theta := (\theta (t), \, 0 \leq t \leq T)$ are $\mathbb{F}$-progressively measurable processes.


On the other hand, the stream of liabilities $P(t) := (P (t), \, 0 \leq t \leq T)$ depends on the past value of the portfolio by the following:  
\begin{equation}
    \label{liab}
 P(t)  = \gamma \sup_{s \in [0,t]} \left\{ X(s) \right\} dt + \int_0^t        X (s - \delta) d A(s).
 \end{equation}
The first term models a guaranteed withdrawal amount as a fraction $\gamma \in  (0,1)$ of the running maximum value of the investment value. Instead, the second term models a Stieltjes integral representing the total amount of continuous claims that depends on a past value of the investment and that are triggered by the increasing continuous function $A$. We emphasize that if we consider  no dependence on the value of the investment $X$, i.e. only $\int_0^t         d A(s)$, we obtain the well-known case with $A$ representing a cumulative
consumption process. See, e.g., \cite{ka-pe-qu/97} for a detailed description or \cite{TouEli} for the problem of utility maximization under a drawdown constraint setting.

We consider a  self financing investment portfolio $X := \left( X (t), 0 \leq t \leq T \right)$, while the admissible strategy $\pi := \left( \pi (t), 0 \leq t \leq T \right)$ denotes the amount invested in the risky bond $D$. 


We denote $\mu(t) = r(t) + \theta (t) \sigma(t)$ and we write the dynamic of $X$ by the following SDE 

\begin{align}
\label{app_for}
\displaystyle dX (t) =  &\pi (t) \dfrac{dD(t)}{D(t)} + \left( X (t) - \pi (t) \right) \dfrac{dS_0 (t)}{S_0(t)} - dP(t)\medskip \\ \notag
\displaystyle \quad \qquad =  &\pi (t) \left( \mu (t) dt + \sigma (t) d W(t) \right) + \left( X (t) - \pi (t) \right) r(t) dt\medskip\\ \notag \displaystyle
& - \gamma \sup_{s \in [0,t]} \left\{ X(s) \right\} dt - \int_0^t        X (s - \delta) d A(s)\medskip \\ \notag \displaystyle
X (T) =  & C a (T) \medskip \, , \notag
\end{align}
$a$ being the annuity factor $
a(T) = \mathbb{E}^{\mathbb{Q}} \left[ \int_T^\infty e^{-\int_T^s r(u)du} ds \vert \mathcal{F}_T \right]
$.

Equation \eqref{app_for} models a variable annuity contract where the policyholder's contributions are invested into two assets ($D$ and $S_0$). Positive returns are distributed to policyholder account based on on the maximum value of the investment and on a prescribed process $A$ (hedging fee) while the remaining value at maturity is  received as a life-time annuity.

From \cite{de/11}, we know that there exists a unique equivalent martingale measure $\mathbb{Q} \sim \mathbb{P}$ under which the discounted price process $S$ is a $(\mathbb{Q},\mathbb{F})$-martingale. Thus, we perform the following change of variables 

\begin{equation}
    Y(t) = X(t) e^{-\int_0^t r(s) ds}\, , \qquad Z(t) = \pi(t) \sigma(t) e^{-\int_0^t r(s) ds} \, , \qquad 0 \leq t \leq T 
\end{equation}
giving the following dynamic for the discounted portfolio process $Y := (Y(t))_{0 \leq t \leq T}$ under the measure $\mathbb{Q}$
\begin{equation}
\label{app_for2}
\begin{array}{l}
    \displaystyle Y(t) = C \tilde{a} (T)  + \int_t^T \gamma \sup_{u \in [0,s] } \left\{ Y (u) e^{-\int_u^s r(v) dv}\right\} ds + \medskip \\ \displaystyle \hspace{5cm} +  \int_t^T Y (s - \delta) e^{-\int_0^{s - \delta} r(v) dv} dA(s) - \int_t^T Z(s)d W^\mathbb{Q} (s) \, ,\medskip
    \end{array}
\end{equation}
$W^\mathbb{Q}$ being a $\mathbb{Q}$-Brownian motion. 

By assuming that conditions
$\mathrm{(A}_{0}\mathrm{)}$ - $\mathrm{(A}%
_{3}\mathrm{)}$ and
$\mathrm{(H}_{1}\mathrm{)}$ - $\mathrm{(H}_{2}\mathrm{)}$ hold true and
by applying Theorem \ref{theorem 2}, we obtain existence and uniqueness of the solution of  equation \eqref{app_for2}. Moreover, the stability of the investment under a perturbation (in uniform norm) of the distribution of the prescribed cumulative distribution is obtained by  Theorem \ref{theorem 2'}, letting to model a robust hedging for the investment with respect to a modification of the prescribed cumulative distribution of future claims.




\section{Conclusions and further developments}

In this article we develop a theoretical framework to study a BSDE with time-delayed generator whose dynamic depends also on  Stieltjes integral term. Under regular assumptions of the coefficients and small delay, we prove the well-posedness of the problem in terms of existence, uniqueness and stability under a perturbation in uniform norm. We also provide an application of our results for a BSDE in insurance setting. Moreover, we obtain the global (in time) well posedness of the BSDE for an arbitrary delay that represent a novel result
in the literature, representing a first attempt to handle (globally) time delayed BSDE. Providing a solid theoretical background for this setting could open up new directions for applications.


Concerning further direction of research, other extensions  would consider the forward reflected SDE linked to the Stieltjes integral in \eqref{BSDE_0} to investigate the corresponding FBSDE with delayed generator and possible connections with the nonlinear 
PDE with Neumann
boundary conditions in the spirit of \cite{pa-zh/98}. Another possibility concerns  considering 
Stieltjes integration with respect to increasing functions that are not necessarily continuous, dealing with dynamics driven by Poisson random measure.


\appendix

\section{Appendix}

In this section we state the result used in the proof of Theorem
\ref{theorem 2'}. It is a variant of the Helly-Bray theorem for the stochastic
case and is also stronger than Proposition 3.4 from \cite{za/08}:\footnote{In
the same time, it corrects an error in the statement of that result:
\textquotedblleft Let $X_{n}$, $K_{n}$ $:(%
\Omega
n,\mathcal{F}n,Pn)\rightarrow\mathbf{W},n\geq1,$ be two sequences of random
variables, converging in distribution to $X$, respectively $K$%
\textquotedblright, should be replaced with \textquotedblleft Let
$(X_{n},K_{n}):(%
\Omega
n,\mathcal{F}n,Pn)\rightarrow\mathbf{W}^{2},n\geq1,$ be a sequence of random
variables, converging in distribution to $(X,K)$\textquotedblright. We
emphasize that this doesn't affect in any way the validity of the other
results in that paper, since the arguments involved use in fact this stronger
assumption.}

\begin{proposition}
\label{prop1}Let $(X_{n},H_{n}):(%
\Omega
_{n},\mathcal{F}_{n},\mathbb{P}_{n})\rightarrow C([0,T];\mathbb{R}^{d}%
)^{2},\ n\geq1,$ be a sequence of random variables, converging in distribution
to a random variable $(X,H):(%
\Omega
,\mathcal{F},\mathbb{P})\rightarrow C([0,T];\mathbb{R}^{d})^{2}$. If for all
$n\geq1$, $H_{n}$ is $\mathbb{P}_{n}$-a.s. with bounded variation and%
\begin{equation}
\lim_{\nu\rightarrow+\infty}\sup_{n\geq1}\mathbb{P}_{n}\left(  \left\Vert
H_{n}\right\Vert _{BV([0,T];\mathbb{R}^{d})}>\nu\right)  =0,
\label{condition_bv_bounded}%
\end{equation}
then $H$ is $\mathbb{P}$-a.s. with bounded variation and the sequence of
$C[0,T]$-valued random variables $\left(  \int_{0}^{\cdot}\left\langle
X_{n}(s),dH_{n}(s)\right\rangle \right)  _{n\geq1}$ converges in distribution
to $\int_{0}^{\cdot}\left\langle X(s),dH_{n}(s)\right\rangle $.
\end{proposition}

As expected, the proof of this result uses a deterministic Helly-Bray type
theorem aiming uniform convergence. For the reader's convenience, we will
state and prove this result:

\begin{lemma}
\label{lemma}Let $\left(  \boldsymbol{x}_{n}\right)  _{n\geq1}\subseteq
C([0,T];\mathbb{R}^{d})$ and $\left(  \boldsymbol{\eta}_{n}\right)  _{n\geq
1}\subseteq BV([0,T];\mathbb{R}^{d})$ be two sequences of functions such that:

\begin{description}
\item[i)] $\boldsymbol{x}_{n}$ converges uniformly to a function
$\boldsymbol{x} \,\mathbf{\in} \,C([0,T];\mathbb{R}^{d})$;

\item[ii)] $\boldsymbol{\eta}_{n}$ converges uniformly to a function
$\boldsymbol{\eta}$;

\item[iii)] $\sup_{n\geq1}\left\Vert \boldsymbol{\eta}_{n}\right\Vert
_{BV([0,T];\mathbb{R}^{d})}<+\infty$.
\end{description}

Then $\boldsymbol{\eta}\in BV([0,T];\mathbb{R}^{d})$, $\left\Vert
\boldsymbol{\eta}\right\Vert _{BV([0,T];\mathbb{R}^{d})}\leq\liminf
_{n\rightarrow\infty}\left\Vert \boldsymbol{\eta}_{n}\right\Vert
_{BV([0,T];\mathbb{R}^{d})}$ and the sequence of continuous functions $\left(
\int_{0}^{\cdot}\left\langle \boldsymbol{x}_{n}(s),d\boldsymbol{\eta}%
_{n}(s)\right\rangle \right)  _{n\geq1}$ converges uniformly to $\int%
_{0}^{\cdot}\left\langle \boldsymbol{x}(s),d\boldsymbol{\eta}(s)\right\rangle
$.
\end{lemma}

\begin{proof}
The first two assertions are well-known, so we skip their proof.

Let us prove the last one. We say that a tuple $\pi=(t_{0},\dots,t_{k})$ is a
\emph{partition} of $[0,T]$ if $0=t_{0}<t_{1}<\dots<t_{k_{N}}=T$.

We consider $\pi^{N}=(t_{0}^{N},\dots,t_{k_{N}}^{N})$, $N\in\mathbb{N}^{\ast}$
partitions of the interval $[0,T]$ such that%
\[
\lim_{N\rightarrow\infty}\sup_{0\leq i<t_{k_{N}}^{N}}\left\vert t_{i+1}%
^{N}-t_{i}^{N}\right\vert =0.
\]
Let $\boldsymbol{x}^{N}:[0,T]\rightarrow\mathbb{R}^{d}$ be a step-function
approximating $\boldsymbol{x}$, defined by%
\[
\boldsymbol{x}^{N}:=\mathbf{1}_{\{0\}}\boldsymbol{x}(0)+\sum\nolimits_{i=1}%
^{k_{N}}\mathbf{1}_{(t_{i-1},t_{i}]}\boldsymbol{x}(t_{i}).
\]
Let $M:=\sup_{n\geq1}\left\Vert \boldsymbol{\eta}_{n}\right\Vert
_{BV([0,T];\mathbb{R}^{d})}$. Then%
\[%
\begin{array}
[c]{l}%
\displaystyle\left\vert \int_{0}^{t}\left\langle \boldsymbol{x}_{n}%
(s),d\boldsymbol{\eta}_{n}(s)\right\rangle -\int_{0}^{t}\left\langle
\boldsymbol{x}(s),d\boldsymbol{\eta}(s)\right\rangle \right\vert
\leq\left\vert \int_{0}^{t}\left\langle \boldsymbol{x}_{n}(s)-\boldsymbol{x}%
(s),d\boldsymbol{\eta}_{n}(s)\right\rangle \right\vert \medskip\\
\displaystyle\quad+\left\vert \int_{0}^{t}\left\langle \boldsymbol{x}%
(s)-\boldsymbol{x}^{N}(s),d(\boldsymbol{\eta}_{n}-\boldsymbol{\eta
})(s)\right\rangle \right\vert +\left\vert \int_{0}^{t}\left\langle
\boldsymbol{x}^{N}(s),d(\boldsymbol{\eta}_{n}-\boldsymbol{\eta}%
)(s)\right\rangle \right\vert \medskip\\
\displaystyle\leq\left\Vert \boldsymbol{x}_{n}-\boldsymbol{x}\right\Vert
_{C([0,T];\mathbb{R}^{d})}\mathrm{V}_{0}^{T}(\boldsymbol{\eta}_{n})+\left\Vert
\boldsymbol{x}^{N}-\boldsymbol{x}\right\Vert _{C([0,T];\mathbb{R}^{d})}\left(
\mathrm{V}_{0}^{T}(\boldsymbol{\eta}_{n})+\mathrm{V}_{0}^{T}(\boldsymbol{\eta
})\right)  \medskip\\
\displaystyle\quad+\sum\nolimits_{i=1}^{k_{N}}\left\vert \boldsymbol{x}%
(t_{i})\right\vert \cdot\left\vert (\boldsymbol{\eta}_{n}-\boldsymbol{\eta
})(t_{i}\wedge t)-(\boldsymbol{\eta}_{n}-\boldsymbol{\eta})(t_{i-1}\wedge
t)\right\vert .
\end{array}
\]
Therefore,%
\[%
\begin{array}
[c]{l}%
\displaystyle\sup_{t\in\lbrack0,T]}\left\vert \int_{0}^{t}\left\langle
\boldsymbol{x}_{n}(s),d\boldsymbol{\eta}_{n}(s)\right\rangle -\int_{0}%
^{t}\left\langle \boldsymbol{x}(s),d\boldsymbol{\eta}(s)\right\rangle
\right\vert \leq M\left\Vert \boldsymbol{x}_{n}-\boldsymbol{x}\right\Vert
_{C([0,T];\mathbb{R}^{d})}\medskip\\
\displaystyle\quad+2M\left\Vert \boldsymbol{x}^{N}-\boldsymbol{x}\right\Vert
_{C([0,T];\mathbb{R}^{d})}+2\left(  \sum\nolimits_{i=1}^{k_{N}}\left\vert
\boldsymbol{x}(t_{i})\right\vert \right)  \left\Vert \boldsymbol{\eta}%
^{n}-\boldsymbol{\eta}\right\Vert _{C([0,T];\mathbb{R}^{d})}.
\end{array}
\]
It follows that%
\[
\limsup_{n\rightarrow\infty}\sup_{t\in\lbrack0,T]}\left\vert \int_{0}%
^{t}\left\langle \boldsymbol{x}_{n}(s),d\boldsymbol{\eta}_{n}(s)\right\rangle
-\int_{0}^{t}\left\langle \boldsymbol{x}(s),d\boldsymbol{\eta}(s)\right\rangle
\right\vert \leq2M\left\Vert \boldsymbol{x}^{N}-\boldsymbol{x}\right\Vert
_{C([0,T];\mathbb{R}^{d})}.
\]

Since $\lim_{N\rightarrow\infty}\left\Vert \boldsymbol{x}^{N}-\boldsymbol{x}%
\right\Vert _{C([0,T];\mathbb{R}^{d})}=0$, we finally get%
\[
\lim_{n\rightarrow\infty}\sup_{t\in\lbrack0,T]}\left\vert \int_{0}%
^{t}\left\langle \boldsymbol{x}_{n}(s),d\boldsymbol{\eta}_{n}(s)\right\rangle
-\int_{0}^{t}\left\langle \boldsymbol{x}(s),d\boldsymbol{\eta}(s)\right\rangle
\right\vert =0.
\]

\end{proof}

Let us now proceed with the proof of the main result of this section, which
follows the same steps as that of Proposition 3.4 from \cite{za/08}.

\begin{proof}
[Proof of the proposition \ref{prop1}]Let $\mathbf{W}:=C([0,T];\mathbb{R}%
^{d})$, $\mathbf{V}:=C([0,T];\mathbb{R}^{d})\cap BV([0,T];\mathbb{R}^{d})$
and, for $\nu>0$,%
\[
\mathbf{V}_{\nu}:=\left\{  \boldsymbol{\eta}\in\mathbf{V}\mid\left\Vert
\boldsymbol{\eta}\right\Vert _{BV([0,T];\mathbb{R}^{d})}\leq\nu\right\}  .
\]
By the first part of Lemma \ref{lemma}, $\mathbf{V}_{\nu}$ is a closed subset
of the Banach space $\mathbf{W}$.

Let us consider the function $\Lambda:\mathbf{W}\times\mathbf{W}%
\rightarrow\mathbf{W}$ defined by%
\[
\Lambda(\boldsymbol{x},\boldsymbol{\eta})(t):=\left\{
\begin{array}
[c]{ll}%
\int_{0}^{t}\left\langle \boldsymbol{x}(s),d\boldsymbol{\eta}(s)\right\rangle
, & (x,\boldsymbol{\eta})\in\mathbf{W}\times\mathbf{V};\\
0, & (x,\boldsymbol{\eta})\in\mathbf{W}\times(\mathbf{W}\setminus\mathbf{V}).
\end{array}
\right.
\]
By the last conclusion of Lemma \ref{lemma}, the restriction $\Lambda%
\vert\big.%
_{\mathbf{W}\times\mathbf{V}_{\nu}}$ is continuous.

Let now $R_{n}:=\mathbb{P}^{n}\circ(X^{n},H^{n})^{-1}$ and $R_{0}%
:=\mathbb{P}\circ(X,H)^{-1}$, the distribution probabilities of $(X^{n}%
,H^{n})$, respectively $(X,H)$. By the assumptions of the theorem, $\left(
R_{n}\right)  _{n\geq1}$ converges weakly to $R_{0}$, \textit{i.e.}%
\begin{equation}
\lim_{n\rightarrow\infty}\int_{\mathbf{W}\times\mathbf{W}}\Phi\left(
\boldsymbol{x},\boldsymbol{\eta}\right)  R_{n}(d\boldsymbol{x}%
,d\boldsymbol{\eta})=\int_{\mathbf{W}\times\mathbf{W}}\Phi\left(
\boldsymbol{x},\boldsymbol{\eta}\right)  R_{0}(d\boldsymbol{x}%
,d\boldsymbol{\eta}), \label{conv_law}%
\end{equation}
for every bounded continuous functional $\Phi:\mathbf{W}\times\mathbf{W}%
\rightarrow\mathbb{R}$.

First of all, by Portmanteau lemma,%
\[
\limsup_{n\rightarrow\infty}R_{n}\left(  \mathbf{W}\times\mathbf{V}_{\nu
}\right)  \leq R_{0}\left(  \mathbf{W}\times\mathbf{V}_{\nu}\right)
,\ \forall\nu>0.
\]
Since, by condition (\ref{condition_bv_bounded}),%
\begin{equation}
\lim_{\nu\rightarrow+\infty}\inf_{n\geq1}R_{n}\left(  \mathbf{W}%
\times\mathbf{V}_{\nu}\right)  =1, \label{condition_bv_bounded2}%
\end{equation}
we get $\lim_{\nu\rightarrow+\infty}R_{0}\left(  \mathbf{W}\times
\mathbf{V}_{\nu}\right)  =1$, \textit{i.e.} $R_{0}\left(  \mathbf{W}%
\times\mathbf{V}\right)  =1$, meaning that $H$ is $\mathbb{P}$-a.s. of bounded variation.

Let now $\phi:C[0,T]\rightarrow\mathbb{R}$ be an arbitrary bounded continuous
functional. It remains to prove that $\lim_{n\rightarrow\infty}\mathbb{E}%
\phi\left(  \Lambda(X^{n},H^{n})\right)  =\mathbb{E}\phi\left(  \Lambda
(X,H)\right)  $, which can be written as%
\[
\lim_{n\rightarrow\infty}\int_{\mathbf{W}\times\mathbf{W}}(\phi\circ
\Lambda)dR_{n}=\int_{\mathbf{W}\times\mathbf{W}}(\phi\circ\Lambda)dR_{0}.
\]
Since $\phi\circ\Lambda%
\vert\big.%
_{\mathbf{W}\times\mathbf{V}_{\nu}}$ is bounded and continous, it can be
extended to a continuous functional $\Phi_{\nu}:\mathbf{W}\times
\mathbf{W}\rightarrow\mathbb{R}$, bounded by $M:=\sup_{\mathbf{z}\in
C[0,T]}\phi(\mathbf{z})$; hence, by (\ref{conv_law}),%
\[
\lim_{n\rightarrow\infty}\int_{\mathbf{W}\times\mathbf{W}}\Phi_{\nu}\left(
\boldsymbol{x},\boldsymbol{\eta}\right)  R_{n}(d\boldsymbol{x}%
,d\boldsymbol{\eta})=\int_{\mathbf{W}\times\mathbf{W}}\Phi_{\nu}\left(
\boldsymbol{x},\boldsymbol{\eta}\right)  R_{0}(d\boldsymbol{x}%
,d\boldsymbol{\eta}).
\]
Let us estimate the term%
\[
T_{n,\nu}:=\left\vert \int_{\mathbf{W}\times\mathbf{W}}\left(  \Phi_{\nu
}\left(  \boldsymbol{x},\boldsymbol{\eta}\right)  -\phi\circ\Lambda\right)
R_{n}(d\boldsymbol{x},d\boldsymbol{\eta})\right\vert ,
\]
for $n\in\mathbb{N}$ (including then the case $n=0$). We have%
\begin{align*}
T_{n,\nu}  &  \leq\int_{\mathbf{W}\times\mathbf{W}}\left\vert \Phi_{\nu
}\left(  \boldsymbol{x},\boldsymbol{\eta}\right)  -\phi\circ\Lambda\right\vert
R_{n}(d\boldsymbol{x},d\boldsymbol{\eta})=\int_{\mathbf{W}\times
(\mathbf{W\setminus V}_{\nu})}\left\vert \Phi_{\nu}\left(  \boldsymbol{x}%
,\boldsymbol{\eta}\right)  -\phi\circ\Lambda\right\vert R_{n}(d\boldsymbol{x}%
,d\boldsymbol{\eta})\\
&  \leq2MR_{n}\left(  \mathbf{W}\times(\mathbf{W\setminus V}_{\nu})\right)
=2M\left(  1-R_{n}\left(  \mathbf{W}\times\mathbf{V}_{\nu}\right)  \right)  .
\end{align*}
Hence, by (\ref{condition_bv_bounded2}) and its consequence%
\[
\lim_{\nu\rightarrow+\infty}\sup_{n\geq0}T_{n,\nu}=0.
\]
Finally, for all $n\geq1$ and $\nu>0$,%
\begin{multline*}
\left\vert \int_{\mathbf{W}\times\mathbf{W}}(\phi\circ\Lambda)dR_{n}%
-\int_{\mathbf{W}\times\mathbf{W}}(\phi\circ\Lambda)dR_{0}\right\vert \\
\leq\left\vert \int_{\mathbf{W}\times\mathbf{W}}\Phi_{\nu}\left(
\boldsymbol{x},\boldsymbol{\eta}\right)  R_{n}(d\boldsymbol{x}%
,d\boldsymbol{\eta})-\int_{\mathbf{W}\times\mathbf{W}}\Phi_{\nu}\left(
\boldsymbol{x},\boldsymbol{\eta}\right)  R_{0}(d\boldsymbol{x}%
,d\boldsymbol{\eta})\right\vert +T_{n,\nu}+T_{0,\nu}%
\end{multline*}
and therefore%
\[
\limsup_{n\rightarrow\infty}\left\vert \int_{\mathbf{W}\times\mathbf{W}}%
(\phi\circ\Lambda)dR_{n}-\int_{\mathbf{W}\times\mathbf{W}}(\phi\circ
\Lambda)dR_{0}\right\vert \leq2\sup_{n\geq0}T_{n,\nu},\ \forall\nu>0
\]
which, by passing to the limit as $\nu\rightarrow0$, yields the desired conclusion.
\end{proof}

\addcontentsline{toc}{section}{References}

\end{document}